\documentclass[11pt]{amsart}
\usepackage{palatino}
\usepackage{amsfonts}
\usepackage{amssymb}
\usepackage{amscd}
\usepackage[breaklinks,bookmarksopen,bookmarksnumbered]{hyperref}
\usepackage[dvips]{graphicx}

\newtheorem{theorem}{Theorem}[section]
\newtheorem{proposition}[theorem]{Proposition}

\newtheorem{lemma}[theorem]{Lemma}
\theoremstyle{definition}

\newtheorem{remark}[theorem]{Remark}

\newtheorem{conjecture/question}[theorem]{Conjecture/Question}

\newtheorem{remark/definition}[theorem]{Remark/Definition}
\newtheorem{terminology/notation}[theorem]{Terminology/Notation}

\def\GG{{\textbf G}}

\def\PP{{\textbf P}}
\def\OO{\mathcal{O}}

\def\cD{\mathcal{D}}

\def\F{\mathcal{F}}
\def\P{\mathcal{P}}

\def\cS{\mathcal{S}}

\def\cM{\mathcal{M}}

\def\cU{\mathcal{U}}

\def\Pic0{{\rm Pic}^0(X)}

\def\ff{\overline{\mathcal{F}}}
\def\mm{\overline{\mathcal{M}}}
\def\ss{\overline{\mathcal{S}}}

\def\dd{\overline{\mathcal{D}}}

\def\thet{\overline{\Theta}_{\mathrm{null}}}

\begin{document}
\title{The intermediate type of certain moduli spaces of curves}

\author[G. Farkas]{Gavril Farkas}

\address{Humboldt-Universit\"at zu Berlin, Institut F\"ur Mathematik,  Unter den Linden 6
\hfill \newline\texttt{}
 \indent 10099 Berlin, Germany} \email{{\tt farkas@math.hu-berlin.de}}
\thanks{}

\author[A. Verra]{Alessandro Verra}
\address{Universit\'a Roma Tre, Dipartimenti di Matematica, Largo San Leonardo Murialdo \hfill
\indent 1-00146 Roma, Italy}
 \email{{\tt
verra@mat.unirom3.it}}
\thanks{Research of the first author partially supported by Sonderforschungsbereich  "Raum-Zeit-Materie".}

\maketitle

A well-established principle of Mumford  asserts that all  moduli spaces of curves of genus $g>2$ (with or without marked points or level structure), are varieties of general type, except a finite number of cases occurring for relatively small genus, when these varieties tend to be unirational, or at least uniruled, see \cite{HM}, \cite{EH1}, \cite{FL}, \cite{F3}, \cite{Log}, \cite{V} for illustrations of this fact. In all known cases, the transition from uniruledness to being of general type is quite sudden and until now no examples were known of naturally defined moduli spaces of curves of intermediate Kodaira dimension.
The aim of this paper is to discuss the very surprising birational geometry of special moduli spaces of curves, which in particular have intermediate Kodaira dimension.

The moduli space $\cS_g$ of smooth spin curves parameterizes pairs
$[C, \eta]$, where $[C]\in \cM_g$ is a curve of genus $g$ and
$\eta\in \mbox{Pic}^{g-1}(C)$ is a theta-characteristic. The map $\pi: \cS_g \rightarrow \cM_g$ is an \'etale covering of degree $2^{2g}$ and
$\cS_g$ is a disjoint union of two connected components $\cS_g^{+}$
and $\cS_g^{-}$ of relative degrees $2^{g-1}(2^g+1)$ and
$2^{g-1}(2^g-1)$ corresponding to even and odd theta-characteristics
respectively. We denote by $\ss_g$ the Cornalba compactification of $\cS_g$,  that is, the coarse moduli space of the stack of
stable spin curves of genus $g$, cf. \cite{C}. The projection $\pi: \cS_g \rightarrow \cM_g$ extends to a
finite covering $\pi:\ss_g\rightarrow \mm_g$ branched along the boundary divisor $\Delta_0$ of $\mm_g$.
It is known that $\ss_g^+$ is a variety of general type for $g>8$ and uniruled for $g<8$, cf. \cite{F3}. We show that the only remaining case, that of $\ss_8^+$,
gives rise to a variety of Calabi-Yau type:

\begin{theorem}\label{spin8}
The Kodaira dimension of $\ss_8^+$ is equal to zero.
\end{theorem}

We point out that the Kodaira dimension of the odd spin moduli space $\ss_g^-$ is known for all  genera $g$, cf.  \cite{FV}. Thus $\ss_g^-$ is uniruled for $g\leq 11$ (even unirational for $g\leq 9$), and of general type for $g\geq 12$. In particular, we observe the surprising phenomenon that $\ss_8^-$ is unirational, whereas $\ss_8^+$ is of Calabi-Yau type!
\vskip 4pt

The proof of Theorem \ref{spin8} relies on two main ideas: Following \cite{F3}, one finds an \emph{explicit} effective representative for the canonical divisor $K_{\ss_8^+}$ as a
$\mathbb Q$-combination of the divisor $\overline{\Theta}_{\mathrm{null}}\subset \ss_8^+$ of vanishing theta-nulls, the pull-back $\pi^*(\mm_{8, 7}^2)$ of the Brill-Noether divisor $\mm_{8, 7}^2$ on $\mm_8$ of curves with a $\mathfrak g^2_7$, and boundary divisor classes corresponding to spin curves whose underlying stable model is of compact type. Each irreducible component of this particular representative of $K_{\ss_8^+}$ is rigid (see Section 1). Then we use in an essential way the existence of a Mukai model of $\mm_8$ as a GIT quotient of a bundle over the Grassmannian $\GG:=G(2, 6)$ cf. \cite{M2}, in order to prove the following result:

\begin{proposition}\label{pencil8}
The uniruled divisor $\thet\subset \ss_8^+$ is swept by rational curves $R\subset \ss_8^+$ such that $R\cdot \thet=-1$ and $R\cdot \pi^*(\mm_{8, 7}^2)=0$. Furthermore $R$ is disjoint from all boundary divisors $A_i, B_i\subset \ss_8^+$ for $i=1, \ldots, 4$.
\end{proposition}
The pencil $R$ corresponds to spin curves lying on special doubly elliptic $K3$ surfaces $S$, chosen in such a way that the rank $3$ quadric containing the underlying canonical curve $C\subset \PP^{7}$ corresponding to a general point $[C, \eta]\in \thet$, lifts to a rank $4$ quadric in $\PP^8$ containing the $K3$ surface $S\supset C$. The existence of such $K3$ extensions of $C$ follows from a precise description of  quadrics containing the Pl\"ucker embedding of the Grassmannian $\GG\subset \PP^{14}$  (see Sections 2 and 3). Proposition \ref{pencil8} implies that $K_{\ss_8^+}$ expressed as a weighted sum of $\thet$, the pull-back $ \pi^*(\mm_{8, 7}^2)$ and boundary divisors $A_i, B_i, i=1, \ldots, 4$, is rigid as well. Equivalently, $\kappa(\ss_8^+)=0$.
\vskip 3pt

Our next result concerns the moduli space $\mm_{g, n}$ of stable $n$-pointed curves of genus $g$. For a given genus $g\geq 2$, we define the numerical invariant
$$\zeta(g):=\mathrm{min}\{n\in \mathbb Z_{\geq 0}: \kappa(\mm_{g, n})\geq 0\}.$$
We think of $\zeta(g)$ as measuring the \emph{complexity} of the general curve of genus $g$. Since the relative dualizing sheaf of the forgetful map $\mm_{g, n}\rightarrow \mm_{g, n-1}$ is big, it follows that $\mm_{g, n}$ is of general type for $n>\zeta(g)$. Clearly $\zeta(g)=0$ for $g\geq 22$, cf. \cite{HM}, \cite{EH2}. There exist explicit upper bounds for $\zeta(g)$ for $4\leq g\leq 21$, see \cite{Log}, \cite{F2} Theorem 1.10. In particular, it is known
that $\mm_{10, n}$ is uniruled for $n\leq 9$, of general type for $n\geq 11$, whereas $\kappa(\mm_{10, 10})\geq 0$, cf. \cite{FP} Proposition 7.5. In other words, $\zeta(10)=10$. Similarly, it is known that $\mm_{11, n}$ is uniruled for $n\leq 10$ and of general type for $n\geq 12$. Until now, no example of a space $\mm_{g, n}$ ($g\geq 2$) having intermediate type was known.  Perhaps, the most picturesque finding of our study is the following:

\begin{theorem}\label{gen11}
The moduli space $\mm_{11, 11}$ has Kodaira dimension $19$.
\end{theorem}

Note that $\mbox{dim}(\mm_{11, 11})=41$. In particular, Theorem \ref{gen11} determines the value $\zeta(11)=11$, hence $\zeta(11)> \zeta(10)$. This explains, in precise terms, that counter-intuitively, algebraic curves of genus $10$ are more complicated than  curves of genus $11$!
\vskip 2pt

The equality  $\kappa(\mm_{11, 11})=19$ is related to the existence of the Mukai  fibration
$$q_{11}: \mm_{11, 11}\dashrightarrow \ff_{11},$$ over the $19$-dimensional moduli space $\ff_{11}$ of polarized $K3$ surfaces of degree $20$. The map $q_{11}$ associates to a general element $[C, x_1, \ldots, x_{11}]\in \mm_{11, 11}$ the unique $K3$ surface $S$ containing $C$, see \cite{M3}. According to Mukai,  $S$ is precisely the "dual" $K3$ surface to the non-abelian Brill-Noether locus corresponding to vector bundles of rank $2$
 $$S^{\vee}=SU_C(2, K_C, 6):=\{E\in SU_C(2, K_C): h^0(C, E)\geq 7\}.$$

 An analysis of the fibration $q_{11}$ shows that, (i) the divisor $n\dd_{11}$ is a fixed component of the pluri-canonical linear series $|nK_{\mm_{11, 11}}|$ for all $n\geq 1$, and (ii) the difference $K_{\mm_{11, 11}}-\dd_{11}$ is essentially the pull-back of an ample class on $\ff_{11}$.

\vskip 4pt

The proof of Theorem  \ref{gen11} is similar in spirit to the proof of Theorem \ref{spin8}. An important role is played by the effective divisor
$$\cD_g:=\{[C, x_1, \ldots, x_g]\in \cM_{g, g}: h^0\bigl(C, \OO_C(x_1+\cdots+x_g)\bigr)\geq 2\}.$$
The  class of the closure of $\cD_g$ inside $\mm_{g, g}$ is the following, cf. \cite{Log} Theorem 5.4:
$$\dd_g\equiv -\lambda+\sum_{i=1}^g \psi_i-0\cdot \delta_{\mathrm{irr}}-\sum_{i=0}^{[g/2]} \sum_{T\subset \{1, \ldots, g\}} {|\#(T)-i|+1\choose 2}\delta_{i: T}\in \mbox{Pic}(\mm_{g, g}).$$
Using \cite{FP}, as well as the expression of $K_{\mm_{g, n}}$ in terms of generators of $\mbox{Pic}(\mm_{g, n})$, one finds an explicit  representative of $K_{\mm_{11, 11}}$ as an effective combination of the pull-back to $\mm_{11, 11}$ of the $6$-gonal  divisor $\mm_{11, 6}^1$  on $\mm_{11}$, the divisor $\dd_{11}$, and certain boundary classes $\delta_{i:S}$. We then construct explicit curves $R\subset \mm_{11, 11}$ passing through a general point of $\dd_{11}$, such that $-R\cdot \dd_{11}>0$ equals precisely the multiplicity of $\dd_{11}$ in the above mentioned expression of $K_{\mm_{11, 11}}$. More generally, we show the following:
 \begin{theorem} For $g\leq 11$, the effective divisor $\dd_g\in \mathrm{Eff}(\mm_g)$ is extremal and rigid.
 \end{theorem}
 In genus $11$, using the existence of the above mentioned Mukai fibration, this eventually leads to the equality $\kappa(\mm_{11, 11})=\kappa(\mm_{11}, \mm_{11, 6}^1)=19$, where the last symbol stands for the Iitaka dimension of the linear system $|\mm_{11, 6}^1|$ generated by the Brill-Noether divisors on $\mm_{11}$.

In the final section of this paper we study the uniruledness of $\mm_{g, n}$ when $g\leq 8$.
\begin{theorem}\label{genul8}

\noindent \begin{enumerate}
\item $\mm_{5, n}$ is uniruled for $n\leq 14$ and of general type for $n\geq 15$. In particular, $\zeta(5)=15$.
\item $\mm_{7, n}$ is uniruled for $n\leq 13$, of general type for $n\geq 15$, while $\kappa(\mm_{7, 14})\geq 0$. In particular,
$\zeta(7)=14$.
\item $\mm_{8, n}$ is uniruled for $n\leq 12$ and  of general type for $n\geq 14$. Thus $\zeta(8)\in \{13, 14\}$.
\end{enumerate}
\end{theorem}

New here is the statement about the uniruledness of $\mm_{g, n}$. For the sake of completeness, we have copied from \cite{Log} and \cite{F2} Theorem 1.10, the range in which $\mm_{g, n}$ was known to be of general type when $g=5, 7, 8$. The problem of determining $\zeta(g)$ is thus completely solved for $g\leq 7$. In order to prove Theorem \ref{genul8}, it suffices to establish that $K_{\mm_{g, n}}$ is not pseudo-effective. This is carried out by exhibiting \emph{two extremal uniruled} divisors on $\mm_{g, n}$, satisfying certain numerical properties, cf. Proposition \ref{uniruled}. This simple uniruledness principle seems to be of particular use when studying the birational geometry of moduli spaces, where one typically has an ample supply of explicit extremal effective divisors.

\section{Spin curves and the divisor $\thet$}

We begin by setting notation and terminology. If $\bf{M}$ is a Deligne-Mumford stack, we denote by $\cM$ its associated coarse moduli space. Let $X$ be a complex $\mathbb Q$-factorial variety. A $\mathbb Q$-Weil divisor $D$ on $X$ is said to be \emph{movable} if $\mbox{codim}\bigl(\bigcap_{m} \mbox{Bs}|mD|, X\bigr)\geq 2$, where the intersection is taken over all $m$ which are sufficiently large and divisible. We say that $D$ is \emph{rigid} if $|mD|=\{mD\}$, for all $m\geq 1$ such that $mD$ is an integral Cartier divisor. The \emph{Kodaira-Iitaka dimension} of a divisor $D$ on $X$ is denoted by $\kappa(X, D)$. As usual, we set $\kappa(X):=\kappa(X, K_X)$.

If $D=m_1D_1+\cdots +m_sD_s$ is an effective $\mathbb Q$-divisor on $X$, with irreducible components $D_i\subset X$  and $m_i> 0$ for $i=1, \ldots, s$, a (trivial) way of showing
that $\kappa(X, D)=0$ is by exhibiting for each $1\leq i\leq s$, a curve $\Gamma_i\subset X$ passing through a general point of $D_i$, such that $\Gamma_i\cdot D_i<0$ and $\Gamma_i\cdot D_j=0$ for $i\neq j$.

\vskip 4pt

We recall basic facts about the moduli space $\ss_g^+$ of even spin curves of genus $g$, see \cite{C}, \cite{F3} for details.
An \emph{even spin curve} of genus
$g$ consists of a triple $(X, \eta, \beta)$, where $X$ is a genus
$g$ quasi-stable curve, $\eta\in \mathrm{Pic}^{g-1}(X)$ is a line
bundle of degree $g-1$ such that $\eta_{E}=\OO_E(1)$ for every
rational component $E\subset X$ such that $\#(E\cap (\overline{X-E}))=2$ (such a component is called \emph{exceptional}), and $h^0(X, \eta)\equiv 0
\mbox{ mod } 2$, and finally,  $\beta:\eta^{\otimes
2}\rightarrow \omega_X$ is a sheaf homomorphism which is generically
non-zero along each non-exceptional component of $X$. Even spin curves of genus $g$ form a smooth Deligne-Mumford stack $\pi:\overline{\bf{S}}_g^+\rightarrow \overline{\bf{M}}_g$. At the level of coarse moduli schemes, the morphism $\pi:\ss_g^+\rightarrow \mm_g$ is the stabilization map
$\pi([X, \eta, \beta]):=[\mathrm{st}(X)]$, which associates to a quasi-stable curve its stable model.

We explain the boundary structure of $\ss_g^+$: If $[X, \eta, \beta]\in \pi^{-1}([C\cup_y D])$,
where $[C, y]\in \cM_{i, 1}, [D, y]\in \cM_{g-i, 1}$ and $1\leq i\leq [g/2]$, then necessarily
$X=C\cup_{y_1} E\cup_{y_2} D$, where $E$ is an exceptional
component such that $C\cap E=\{y_1\}$ and $D\cap E=\{y_2\}$.
Moreover $\eta=\bigl(\eta_C, \eta_D, \eta_E=\OO_E(1)\bigr)\in
\mbox{Pic}^{g-1}(X)$, where $\eta_C^{\otimes 2}=K_C, \eta_D^{\otimes
2}=K_D$. The condition $h^0(X, \eta)\equiv 0 \mbox{ mod } 2$,
implies that the theta-characteristics $\eta_C$ and $\eta_D$ have
the same parity. We denote by $A_i\subset \ss_g^+$ the closure of
the locus corresponding to pairs $$([C,y, \eta_C], [D, y, \eta_D])\in
\cS_{i, 1}^+\times \cS_{g-i, 1}^+$$ and by $B_i\subset \ss_g^+$ the
closure of the locus corresponding to pairs $$([C, y, \eta_C], [D, y,
\eta_D])\in \cS_{i, 1}^-\times \cS_{g-i, 1}^{-}.$$

We set
$\alpha_i:=[A_i]\in \mathrm{Pic}(\ss_g^+), \beta_i:=[B_i]\in
\mathrm{Pic}(\ss_g^+)$, and then one has that \begin{equation}\pi^*(\delta_i)=\alpha_i+\beta_i.
\end{equation}

We recall the description of the ramification divisor of the covering $\pi:\ss_g^+\rightarrow \mm_g$. For a point $[X, \eta, \beta]\in \ss_g^+$ corresponding to a stable model $\mbox{st}(X)=C_{yq}:=C/ y\sim q$,
with $[C, y, q]\in \cM_{g-1, 2}$, there are two possibilities
depending on whether $X$ possesses an exceptional component or not.
If $X=C_{yq}$ (i.e. $X$ has no exceptional component) and $\eta_C:=\nu^*(\eta)$ where $\nu:C\rightarrow X$
denotes the normalization map, then $\eta_C^{\otimes 2}=K_C(y+q)$.
For each choice of $\eta_C\in \mathrm{Pic}^{g-1}(C)$ as above, there
is precisely one choice of gluing the fibres $\eta_C(y)$ and
$\eta_C(q)$ such that $h^0(X, \eta) \equiv 0 \mbox{ mod } 2$. We denote by $A_0$ the
closure in $\ss_g^+$ of the locus of spin curves $[C_{yq}, \eta_C\in
\sqrt{K_C(y+q)}]$ as above.

If $X=C\cup_{\{y, q\}} E$, where $E$ is an exceptional component,
then $\eta_C:=\eta\otimes \OO_C$ is a theta-characteristic on $C$.
Since $H^0(X, \omega)\cong H^0(C, \omega_C)$, it follows that $[C,
\eta_C]\in \cS_{g-1}^{+}$.  We denote by
$B_0\subset \ss_g^+$ the closure of the locus of spin curves
$$\bigl[C\cup_{\{y, q\}} E, \ E\cong \PP^1, \ \eta_C\in \sqrt{K_C}, \ \eta_E=\OO_E(1)\bigr]\in \cS_g^+.$$ If
$\alpha_0:=[A_0], \beta_0:=[B_0]\in
\mbox{Pic}(\ss_g^+)$,  we have the relation, see \cite{C}:
\begin{equation}\label{del0}
\pi^*(\delta_0)=\alpha_0+2\beta_0.
\end{equation} In particular, $B_0$ is the ramification divisor of $\pi$.
An important  effective divisor on $\ss_g^{+}$ is the locus of vanishing theta-nulls
$$\Theta_{\mathrm{null}}:=\{[C, \eta]\in \cS_g^{+}: H^0(C, \eta)\neq
0\}.$$ The class of its compactification inside $\ss_g^+$
is given by the formula, cf.
\cite{F3}:
\begin{equation}\label{thetanull}
\overline{\Theta}_{\mathrm{null}}\equiv
\frac{1}{4}\lambda-\frac{1}{16}\alpha_0-\frac{1}{2}\sum_{i=1}^{[g/2]}
 \beta_i\in \   \mathrm{Pic}(\ss_g^{+}).
 \end{equation}
 It is also useful to recall the formula for the canonical class of $\ss_g^+$:
$$K_{\ss_g^{+}}\equiv\pi^*(K_{\mm_g})+\beta_0 \equiv
13\lambda-2\alpha_0-3\beta_0-2\sum_{i=1}^{[g/2]}
(\alpha_i+\beta_i)-(\alpha_1+\beta_1).$$

An argument involving spin curves on certain singular canonical surfaces in $\PP^6$, implies that for $g\leq 9$, the divisor $\thet$ is uniruled and a rigid point in the cone of effective
divisors $\mathrm{Eff}(\ss_g^+)$:

\begin{theorem}\label{extremalthetanull}
For $g\leq 9$ the divisor $\thet\subset \ss_g^+$ is uniruled and
rigid. Precisely, through a general point of $\thet$ there passes a rational curve $\Gamma \subset \ss_g^+$ such that $\Gamma \cdot \thet<0$. In particular, if  $D$ is an effective divisor on $\ss_g^+$ with $D\equiv n\thet$ for some $n\geq 1$, then $D=n\thet$.
\end{theorem}

\begin{proof} We assume $7\leq g\leq 9$, the other cases being similar and simpler. A general point $[C, \eta_C]\in \Theta_{\mathrm{null}}$ corresponds to a canonical
curve $C\stackrel{|K_C|}\hookrightarrow \PP^{g-1}$ lying on a rank $3$ quadric $Q\subset \PP^{g-1}$ such that
$C\cap \mathrm{Sing}(Q)=\emptyset$. The pencil $\eta_C$ is recovered from the ruling of $Q$.

Let
$V\in G\bigl(7, H^0(C, K_C)\bigr)$ be a general subspace such that if $\pi_V:\PP^{g-1}\dashrightarrow \PP(V^{\vee})$ is the projection, then
 $\tilde{Q}:=\pi_V(Q)$ is a quadric of rank $3$. Let $C':=\pi_V(C)\subset \PP(V^{\vee})$ be the projection of the canonical curve $C$.
 By counting dimensions we find that
$$\mbox{dim}\bigl\{I_{C'/\PP(V^{\vee})}(2):=\mbox{Ker}\{\mathrm{Sym}^2(V)\rightarrow H^0(C, K_C^{\otimes 2})\}\bigr\}\geq 31-3g\geq 4,$$
that is, the embedded curve $C'\subset \PP^6$ lies on at least $4$ independent quadrics, namely the rank $3$ quadric $\tilde{Q}$ and $Q_1, Q_2, Q_3\in |I_{C'/\PP(V^{\vee})}(2)|$.
By choosing $V$ sufficiently general we make sure that  $S:=\tilde{Q}\cap Q_1\cap Q_2 \cap Q_3$ is a canonical surface in $\PP(V^{\vee})$ with $8$ nodes corresponding to the intersection $\bigcap_{i=1}^3 Q_i\cap \mathrm{Sing}(\tilde{Q})$ (This transversality statement can also be checked with Macaulay by representing $C$ as a section of the corresponding Mukai variety). From the exact sequence on $S$,
$$0\longrightarrow \OO_S\longrightarrow \OO_S(C)\longrightarrow \OO_C(C)\longrightarrow 0,$$
coupled with the adjunction formula $\OO_C(C)=K_C\otimes K_{S |C}^{\vee}=\OO_C$, as well as the fact $H^1(S, \OO_S)=0$, it follows that
$\mbox{dim }|C|= 1$, that is, $C\subset S$ moves in its linear system. In particular, $\thet$ is a  uniruled divisor for $g\leq 9$.

We determine the numerical parameters of the family $\Gamma \subset \ss_g^+$ induced by varying $C\subset S$. Since $C^2=0$, the pencil $|C|$ is base point free and gives rise to a fibration $f:\tilde{S}\rightarrow \PP^1$, where $\tilde{S}:=\mathrm{Bl}_8(S)$ is the blow-up of the nodes of $S$. This in turn induces a moduli map $m:\PP^1\rightarrow \ss_g^+$ and $\Gamma=:m(\PP^1).$  We have the formulas $$\Gamma \cdot \lambda=m^*(\lambda)=\chi(S, \OO_S)+g-1=8+g-1=g+7,\ \hfill$$ and
$$\Gamma\cdot \alpha_0+2\Gamma \cdot \beta_0=m^*(\pi^*(\delta_0))=m^*(\alpha_0)+2m^*(\beta_0)=c_2(\tilde{S})+4(g-1).$$
Noether's formula gives that
$c_2(\tilde{S})=12\chi(\tilde{S}, \OO_{\tilde{S}})-K_{\tilde{S}}^2=12\chi(S, \OO_S)-K_S^2=80$, hence $m^*(\alpha_0)+2m^*(\beta_0)=4g+76.$ The singular fibres  corresponding to spin curves lying in $B_0$ are those in the fibres over the blown-up nodes and all contribute with multiplicity $1$, that is, $\Gamma \cdot \beta_0=8$ and then $\Gamma \cdot \alpha_0=4g+60$.
It follows that $\Gamma\cdot \thet=-2<0$ (independent of $g$!), which finishes the proof.

To illustrate one of the  cases $g<7$, we discuss the situation on $\ss_4^+$.  We denote by $S=\mathbb F_2$ the blow-up of the vertex of a cone $Q\subset \PP^3$ over a conic in $\PP^3$ and write $\mathrm{Pic}(S)=\mathbb Z\cdot F+\mathbb Z\cdot C_0$, where $F^2=0$, $C_0^2=-2$ and $C_0\cdot F=1$. We choose a Lefschetz pencil of genus $4$ curves in the linear system $|3(C_0+2F)|$. By blowing-up the $18=9(C_0+2F)^2$ base points, we obtain a fibration $f:\tilde{S}:=\mathrm{Bl}_{18}(S)\rightarrow \PP^1$ which induces a family of spin curves $m:\PP^1\rightarrow \ss_4^+$ given by $m(t):=[f^{-1}(t), \OO_{f^{-1}(t)}(F)]$. We have the formulas
$$m^*(\lambda)=\chi(\tilde{S}, \OO_{\tilde{S}})+g-1=4, \   \mbox{ }\mbox{ and }$$
$$m^*(\pi^*(\delta_0))=m^*(\alpha_0)+2m^*(\beta_0)=c_2(\tilde{S})+4(g-1)=34.$$
 The singular fibres lying in $B_0$ correspond to curves in the Lefschetz pencil on $Q$ passing through the vertex of the cone, that is, when $f^{-1}(t_0)$ splits as $C_0+D$, where $D\subset \tilde{S}$ is the residual curve. Since $C_0\cdot D=2$ and $\OO_{C_0}(F)=\OO_{C_0}(1)$, it follows that $m(t_0)\in B_0$. One finds that $m^*(\beta_0)=1$, hence $m^*(\alpha_0)=32$ and
$m^*(\thet)=-1$. Since $\Gamma:=m(\PP^1)$ fills-up the divisor $\thet$, we obtain that $[\thet]\in \mathrm{Eff}(\ss_4^+)$ is rigid.
\end{proof}

\section{Spin curves of genus $8$}

The moduli space $\cM_8$ carries one Brill-Noether divisor, the locus of plane septics $$\cM_{8, 7}^2:=\{[C]\in \cM_8: G_7^2(C)\neq \emptyset\}.$$ The locus $\mm_{8, 7}^2$ is irreducible and for a known constant $c_{8, 7}^2\in \mathbb Z_{>0}$, one has, cf. \cite{EH2},
$$\mathfrak{bn}_8:=\frac{1}{c_{8, 7}^2} \mm_{8, 7}^2\equiv 22\lambda-3\delta_0-14\delta_1-24\delta_2-30\delta_3-32\delta_4\in \mbox{Pic}(\mm_8).$$
In particular, $s(\mm_{8, 7}^2)=6+12/(g+1)$ and this is the minimal slope of an effective divisor on $\mm_8$. The following fact is probably well-known:

\begin{proposition}\label{septics}
Through a general point of $\mm_{8, 7}^2$ there passes a rational curve $R\subset \mm_8$ such that $R\cdot \mm_{8, 7}^2<0$. In particular, the class $[\mm_{8, 7}^2]\in \mathrm{Eff}(\mm_8)$ is rigid.
\end{proposition}
\begin{proof}
One takes a Lefschetz pencil of nodal plane septic curves with $7$ assigned nodes in general position (and $21$ unassigned base points). After blowing up the $21$ unassigned base points as well as the $7$ nodes, we obtain a fibration $f:S:=\mathrm{Bl}_{28}(\PP^2)\rightarrow \PP^1$, and the corresponding moduli map $m:\PP^1\rightarrow \mm_8$ is a covering curve for the irreducible divisor $\mm_{8, 7}^2$. The numerical invariants of this pencil are
$$m^*(\lambda)=\chi(S, \OO_S)+g-1=8\ \mbox{ and } m^*(\delta_0)=c_2(S)+4(g-1)=59,$$
while  $m^*(\delta_i)=0$ for $i=1, \ldots, 4$. We find  $m^*([\mm_{8, 7}^2])=c^2_{8, 7}(8\cdot 22-3\cdot 59)=-c^2_{8, 7}<0$.
\end{proof}

Using (\ref{thetanull}) we find the following explicit representative for the canonical class $K_{\ss_8^+}$:
\begin{equation}\label{can}
K_{\ss_8^+}\equiv \frac{1}{2} \pi^*(\mathfrak{bn}_8)+8\thet+\sum_{i=1}^4 (a_i\ \alpha_i+b_i\ \beta_i),
\end{equation}
where $a_i, b_i>0$ for $i=1, \ldots, 4$. The multiples of each irreducible component appearing in (\ref{can})  are rigid divisors on $\ss_8^+$, but in principle, their sum could still be a movable class. Assuming for a moment Proposition \ref{pencil8}, we explain how this implies Theorem \ref{spin8}:
\vskip 5pt

\noindent \emph{Proof of Theorem \ref{spin8}.} The covering curve $R\subset \thet$ constructed in Proposition \ref{pencil8}, satisfies $R\cdot \thet<0$ as well as $R\cdot \pi^*(\mm_{8, 7}^2)=0$ and $R\cdot \alpha_i=R\cdot \beta_i=0$ for $i=1, \ldots, 4$. It follows from (\ref{can}) that for each $n\geq 1$, one has an equality of linear series on $\ss_8^+$
$$|nK_{\ss_8^+}|=8n\thet+|n(K_{\ss_8^+}-8\thet)|.$$

\noindent Furthermore, from (\ref{can}) one finds  constants $a_i'>0$ for $i=1, \ldots, 4$, such that if $$D\equiv 22\lambda-3\delta_0-\sum_{i=1}^4 a_i'\ \delta_i\in \mbox{Pic}(\mm_8),$$ then the difference $\frac{1}{2}\pi^*(D)-(K_{\ss_8^+}-8\thet)$ is still effective on $\ss_8^+$. We can thus write
$$0\leq \kappa(\ss_8^+)=\kappa\bigl(\ss_8^+, K_{\ss_8^+}-8\thet\bigr)\leq \kappa\bigl(\ss_8^+, \frac{1}{2}\pi^*(D)\bigr)=\kappa\bigl(\ss_8^+, \pi^*(D)\bigr).$$
We claim that $\kappa\bigl(\ss_8^+, \pi^*(D)\bigr)=0$.  Indeed, in the course of the proof of Proposition \ref{septics} we have constructed a covering family $B\subset \mm_8$ for the divisor $\mm_{8, 7}^2$ such that $B\cdot \mm_{8, 7}^2<0$ and $B\cdot \delta_i=0$ for $i=1, \ldots, 4$. We lift $B$ to a family $R\subset \ss_8^+$ of spin curves by taking
$$\tilde{B}:=B\times _{\mm_8} \ss_8^+=\{[C_t, \ \eta_{C_{t}}]\in \ss_8^{+}:
[C_{t}]\in B, \eta_{C_{t}}\in \overline{\mbox{Pic}}^{7}(C_{t}), t
\in \PP^1 \}\subset \ss_8^+.$$ One notes that $\tilde{B}$ is disjoint from the boundary divisors $A_i, B_i\subset \ss_8^+$ for $i=1, \ldots, 4$, hence  $\tilde{B}\cdot \pi^*(D)=2^{g-1}(2^g+1) (B\cdot \mm_{8, 7}^2)_{\mm_8}<0$. Thus we write that
$$\kappa\bigl(\ss_8^+, \pi^*(D)\bigr)=\kappa\bigl(\ss_8^+, \pi^*(D-(22\lambda-3\delta_0)\bigr)=\kappa \bigl(\ss_8^+, \sum_{i=1}^4 a_i'(\alpha_i+\beta_i)\bigr)=0.$$
$ \hfill \Box$

\section{A family of spin curves  $R \subset \ss^+_8$ with $R \cdot \pi^*(\mm_{8,7}^2) = 0$ and $R \cdot \thet = -1$}
The aim of this section is to prove Proposition \ref{pencil8}, which is the key ingredient in the proof of Theorem \ref{spin8}. We begin by reviewing  facts about the geometry of $\mm_8$, in particular the construction of general curves of genus 8 as complete intersections in a rational homogeneous variety, cf. \cite{M2}.

We fix $V\cong \mathbb C^6$ and denote by $\GG:=G(2, V)\subset \PP(\wedge^2 V)$ the Grassmannian of lines. Noting that smooth codimension $7$ linear sections of $\GG$ are canonical curves of genus $8$, one is led to consider the \emph{Mukai model} of the moduli space of curves of genus $8$
$$\mathfrak{M}_8:=G(8, \wedge^2 V)// SL(V).$$
There is a birational map $f:\mm_8\dashrightarrow \mathfrak{M}_8$, whose inverse is given by $f^{-1}(H):=\GG\cap H$, for a general $H\in G(8, \wedge^2 V)$. The map $f$ is constructed as follows:  Starting with a curve $[C]\in \cM_8-\cM_{8, 7}^2$, one notes that $C$ has a finite number of pencils $\mathfrak g^1_5$. We choose $A\in W^1_5(C)$ and set $L:=K_C\otimes A^{\vee}\in W^3_9(C)$. There exists a unique rank $2$ vector bundle $E\in SU_C(2, K_C)$ (independent of $A$!), sitting in an extension
 $$0\longrightarrow A\longrightarrow E\longrightarrow L\longrightarrow 0,$$
 such that $h^0(E)=h^0(A)+h^0(L)=6$. Since $E$ is globally generated, we  define the map
 $$\phi_E: C\rightarrow G\bigl(2, H^0(C, E)^{\vee}\bigr), \ \mbox{    } \ \mbox{ } \phi_E(p):=E(p)^{\vee} \ \bigl(\hookrightarrow H^0(C, E)^{\vee}\bigr),$$ and let $\wp: G(2, H^0(C, E)^{\vee})\rightarrow \PP(\wedge^2 H^0(C, E)^{\vee})$ be the Pl\"ucker embedding.  The determinant map $u:\wedge ^2 H^0(E)\rightarrow H^0(K_C)$ is surjective, that is, $H^0(K_C)^{\vee}\in G(8, \wedge^2 H^0(E)^{\vee})$, see \cite{M2} Theorem C. We set
 $$f([C]):=[C\stackrel{\wp\circ \phi_E}\longrightarrow \PP\bigl(\wedge^2 H^0(E)^{\vee}\bigr), \ \PP(H^0(K_C)^{\vee})] \ \mathrm{ mod }\  SL(V) \in \mathfrak{M}_8.$$
 It follows from \cite{M2} that the exceptional divisors of $f$ are the Brill-Noether locus $\mm_{8, 7}^2$ and the boundary divisors $\Delta_1, \ldots, \Delta_4$. The map $f^{-1}$ does not contract any divisors.
 \vskip 4pt

Inside the moduli space $\F_8$ of polarized $K3$ surfaces $[S, h]$ of degree $h^2=14$, we consider the following \emph{Noether-Lefschetz} divisor
$$\mathfrak{NL}:=\{[S, \OO_S(C_1+C_2)]\in \F_8: \mathrm{Pic}(S)\supset \mathbb Z\cdot C_1\oplus \mathbb Z\cdot C_2,\  \  C_1^2=C_2^2=0, \ C_1\cdot C_2=7\},$$
of doubly-elliptic $K3$ surfaces. For a general element $[S, \OO_S(C)]\in \mathfrak{NL}$, the embedded surface $S\stackrel{|\OO_S(C)|}\hookrightarrow \PP^8$ lies on a rank $4$ quadric whose rulings induce the elliptic pencils $|C_1|$ and $|C_2|$ on $S$. We denote by $\mathfrak{NL}'\subset \mathfrak{NL}$ the open subset corresponding to polarized surfaces $[S, \OO_S(C_1+C_2)]$ such that $\mbox{Pic}(S)=\mathbb Z\cdot C_1\oplus \mathbb Z\cdot C_2$. Then we consider the $\PP^3$-bundle \ $\cU\rightarrow \mathfrak{NL}'$ classifying pairs $\bigl([S, \OO_S(C_1+C_2)], C\subset S\bigr)$, where $$C\in |H^0(S, \OO_S(C_1))\otimes H^0(S, \OO_S(C_2))|\subset |H^0(S, \OO_S(C_1+C_2))|.$$ An element of $\cU$ corresponds to a hyperplane section $C\subset S\subset \PP^8$ of a doubly-elliptic $K3$ surface, such that the intersection of $C$ with the rank $4$ quadric induced by the elliptic pencils, has rank $3$. There exists a rational map $$q:\cU\dashrightarrow \thet, \ \ \ \mbox{ } \ q\bigl([S, \OO_S(C_1+C_2)], C\bigr):=[C, \OO_C(C_1)=\OO_C(C_2)].$$ Clearly $\cU$ is irreducible and $\mbox{dim}(\cU)=21 \bigl(=3+\mbox{dim}(\mathfrak{NL})\bigr)$. We shall show that the morphism $q$ is dominant, by explicitly describing its generic fibre. This produces a parametrization of the divisor $\thet$, in particular it provides an explicit covering curve.
\vskip 5pt

We fix a general point  $[C, \eta]\in \thet\subset \ss_8^+$, with $\eta$ a vanishing theta-null. Then
$$
C \subset Q \subset \PP^7:=\PP\bigl(H^0(C, K_C)^{\vee}\bigr),
$$
where $Q\in H^0(\PP^7, \mathcal{I}_{C/\PP^7}(2))$ is the rank $3$ quadric such that the ruling of $Q$ cuts out on $C$ precisely  $\eta$.
As explained, there exists a linear embedding $\PP^7 \subset \PP^{14}:=\PP\bigl(\wedge^2 H^0(E)^{\vee}\bigr)$  such that
$\PP^7 \cap \GG = C$. The restriction map yields an isomorphism between spaces of quadrics, cf. \cite{M2},
$$\mathrm{res}_C: H^0(\GG, \mathcal{I}_{\GG/\PP^{14}}(2))\stackrel{\cong}\longrightarrow H^0(\PP^7, \mathcal{I}_{C/\PP^7}(2)).$$
In particular there is a unique quadric $\GG\subset \tilde{Q} \subset \PP^{14}$ such that $\tilde{Q} \cap \PP^7=Q$.

There are three possibilities for the rank of any quadric $\tilde{Q}\in H^0(\PP^{14}, \mathcal{I}_{\GG/\PP^{14}}(2))$:
\  (a) $\mathrm{rk}(\tilde{Q})=15$, \ (b) $\mathrm{rk}(\tilde{Q})=6$ and then $\tilde{Q}$ is a \emph{Pl\"ucker quadric}, or \ (c) $\mbox{rk}(\tilde{Q})=10$, in which case $\tilde{Q}$ is a sum of two Pl\"ucker quadrics, see \cite{M2}.
\vskip4pt
\begin{proposition} For a general $[C, \eta]\in \thet$, the quadric $\tilde{Q}$ is smooth, that is, $\mathrm{rk}(\tilde{Q})=15$.
\end{proposition}
\begin{proof} We may assume that $\mbox{dim }G^1_5(C)=0$ (in particular $C$ has no $\mathfrak g^1_4$'s), and  $G^2_7(C)=\emptyset$. The space $\PP(\mathrm{Ker}(u))\subset \PP\bigl(\wedge^2 H^0(E)\bigr)$ is identified with the space of hyperplanes $H\in (\PP^{14})^{\vee}$ containing the canonical space $\PP^7$.
\vskip 3pt

\noindent {\emph{Claim:}} If $\mbox{rk}({\tilde{Q}})<15$, there exists a pencil of $8$-dimensional planes $\PP^7\subset \Xi \subset \PP^{14}$, such that $S:=\GG\cap \Xi$ \ is a  $K3$ surface containing $C$ as a hyperplane section, and
$$
\mathrm{rk}\bigl\{Q_{\Xi}:=\tilde{Q}\cap \Xi\in H^0(\Xi, \mathcal{I}_{S/\Xi(2)})\bigr\}=3.
$$

The conclusion of the claim contradicts the assumption that $[C, \eta]\in \thet$ is general. Indeed, we pick such an $8$-plane $\Xi$ and corresponding $K3$ surface $S$. Since $\mbox{Sing}(Q)\cap C=\emptyset$, where $Q_{\Xi}\cap \PP^7=Q$, it follows that $S\cap \mbox{Sing}(Q_{\Xi})$ is finite. The ruling of $Q_{\Xi}$ cuts out an elliptic pencil $|E|$ on $S$. Furthermore, $S$ has nodes at the points $S\cap \mbox{Sing}(Q_{\Xi})$. For numerical reasons, $\#\mbox{Sing}(S)=7$, and then on the surface $\tilde{S}$ obtained from $S$ by resolving the $7$ nodes, one has the linear equivalence $$C\equiv 2E+\Gamma_1+\cdots +\Gamma_7,$$ where $\Gamma_i^2=-2, \ \Gamma_i\cdot E=1$ for $i=1, \ldots, 7$ and $\Gamma_i\cdot \Gamma_j=0$ for $i\neq j$. In particular $\mbox{rk}(\mbox{Pic}(\tilde{S}))\geq 8$. A standard parameter count, see e.g. \cite{Do}, shows that
 $$\mathrm{dim}\bigl\{(S, C): C\in |\OO_S(2E+\Gamma_1+\cdots+ \Gamma_7)|\bigr\}\leq 19-7+\mbox{dim}|\OO_{\tilde{S}}(C)|=20.$$
 Since $\mbox{dim}(\thet)=20$ and a general curve $[C]\in \thet$ lies on infinitely many such $K3$ surfaces $S$, one obtains a contradiction.

 \vskip 4pt
We are left with proving the claim made in the course of the proof. The key point is to describe the intersection $\PP(\mbox{Ker}(u))\cap \tilde{Q}^{\vee}$, where we recall that the linear span $\langle \tilde{Q}^{\vee} \rangle$ classifies  hyperplanes $H\in (\PP^{14})^{\vee}$ such that
$\mbox{rk}(\tilde{Q}\cap H)\leq \mbox{rk}(\tilde{Q})-1$. Note also that $\mbox{dim } \langle \tilde{Q}\rangle=\mbox{rk}(\tilde{Q})-2$.

If $\mbox{rk}(\tilde{Q})=6$, then $\tilde{Q}^{\vee}$ is contained in the dual Grassmannian $\GG^{\vee}:=G(2, H^0(E))$, cf. \cite{M2} Proposition 1.8. Points in the intersection $\PP(\mbox{Ker}(u))\cap \GG^{\vee}$ correspond to decomposable tensors $s_1\wedge s_2$, with $s_1, s_2\in H^0(C, E)$,  such that $u(s_1\wedge s_2)=0$. The image of the morphism $\OO_C^{\oplus 2}\stackrel{(s_1, s_2)}\longrightarrow E$ is thus a subbundle $\mathfrak g^1_5$ of $E$ and there is a bijection
$$ \PP(\mathrm{Ker}(u))\cap \GG\bigl(2, H^0(E)\bigr)\cong  W^1_5(C).$$ It follows, there are at most finitely many tangent hyperplanes to $\tilde{Q}$ containing the space $\PP^7=\langle C\rangle$, and consequently, $\mbox{dim}\bigl(\PP(\mbox{Ker}(u))\cap \langle \tilde{Q}^{\vee}\rangle\bigr)\leq 1$. Then there exists a codimension $2$ linear space $W^{12}\subset \PP^{14}$ such that $\mbox{rk}(\tilde{Q}\cap W)=3$, which proves the claim (and much more), in the case $\mbox{rk}(\tilde{Q})=6$.

When $\mbox{rk}(\tilde{Q})=10$, using the explicit description of the dual quadric $\tilde{Q}^{\vee}$ provided in \cite{M2} Proposition 1.8, one finds that  $\mbox{dim}\bigl(\PP(\mbox{Ker}(u))\cap \langle \tilde{Q}^{\vee}\rangle\bigr)\leq 4$. Thus there exists a codimension $5$ linear section
$W^9\subset \PP^{14}$ such that $\mbox{rk}(\tilde{Q}\cap W)=3$, which implies the claim when $\mbox{rk}(\tilde{Q})=10$ as well.

\end{proof}

We consider an $8$-dimensional linear extension
$\PP^7\subset \Lambda^8 \subset \PP^{14}$ of the canonical space $\PP^7=\langle C\rangle$, such that
$S_{\Lambda} := \Lambda \cap \GG$
is a smooth K3 surface. The restriction map
$$\mathrm{res}_{C/S_{\Lambda}}: H^0(\Lambda, \mathcal I_{S_{\Lambda}/\Lambda}(2)) \to H^0(\PP^7, \mathcal I_{C / \PP^7}(2))$$
is an isomorphism, cf. \cite{SD}. Thus there exists a \emph{unique}  quadric
$S_{\Lambda}\subset Q_{\Lambda} \subset \Lambda$
with $Q_{\Lambda} \cap \PP^7 = Q$.
Since $\mbox{rk}(Q)= 3$, it follows that $3 \leq \mbox{rk}(Q_{\Lambda}) \leq 5$ and it is easy to see that for a general
$\Lambda$, the corresponding quadric $Q_{\Lambda}\subset \Lambda$ is of rank $5$. We show however, that one can find $K3$-extensions of the canonical curve $C$, which lie on quadrics of rank $4$:

\begin{proposition} For a general $[C, \eta]\in \thet$, there exists a pencil of $8$-dimensional extensions $$\PP(H^0(C, K_C)^{\vee})\subset \Lambda\subset \PP^{14}$$ such that $\mathrm{rk}(Q_{\Lambda})=4$. It follows that there exists a smooth $K3$ surface $S_{\Lambda}\subset \Lambda$ containing $C$ as a transversal hyperplane section, such that $\mathrm{rk}(Q_{\Lambda})=4$.
\end{proposition}
\begin{proof} We pass from projective to vector spaces and view the rank $15$ quadric $$\tilde{Q}: \wedge^2 H^0(C, E)^{\vee}\stackrel{\sim}\longrightarrow \wedge^2 H^0(C, E)$$ as an isomorphism, which by restriction to $H^0(C, K_C)^{\vee}\subset \wedge^2 H^0(C, E)^{\vee}$, induces the rank $3$ quadric
 $Q:H^0(C, K_C)^{\vee}\rightarrow H^0(C, K_C)$. The map $u\circ \tilde{Q}: \wedge^2 H^0(E)^{\vee}\rightarrow H^0(K_C)$ being surjective, its kernel $\mbox{Ker}(u\circ \tilde{Q})$ is a $7$-dimensional vector space containing the $5$-dimensional subspace $\mbox{Ker}(Q)$. We choose an arbitrary element $$[\bar{v}:=v+\mathrm{Ker}(Q)]\in \PP\Bigl(\frac{\mathrm{Ker}(u\circ \tilde{Q})}{\mathrm{Ker}(Q)}\Bigr),$$ inducing a subspace
 $H^0(C, K_C)^{\vee}\subset \Lambda:=H^0(C, K_C)^{\vee}+\mathbb C v \subset \wedge^2 H^0(C, E)^{\vee},$
 with the property that $\mbox{Ker}(Q_{\Lambda})=\mbox{Ker}(Q)$, where $Q_{\Lambda}: \Lambda \rightarrow \Lambda^{\vee}$ is induced from $\tilde{Q}$ by restriction and projection. It follows that $\mbox{rk}(Q_{\Lambda})=4$. Moreover, we have shown that $\mbox{dim } q^{-1}([C, \eta])\leq 1$, in particular $q$ is dominant.
\end{proof}

\vskip 4pt

Now we can begin the proof of  Proposition \ref{pencil8}. Let
$C \subset Q \subset \PP^7$
be a general canonical curve endowed with a vanishing theta-null, where $Q\in H^0\bigl(\PP^7, I_{C/\PP^7}(2)\bigr)$ is the corresponding rank $3$ quadric. We choose a general $8$-plane $\PP^7\subset \Lambda\subset \PP^{14}$ such that
$S:= \Lambda \cap \GG$
is a smooth K3 surface, and the lift of $Q$ to $\Lambda$
$$Q_{\Lambda}\in H^0\bigl(\Lambda, \mathcal{I}_{S/\Lambda}(2)\bigr)$$
has rank $4$. Moreover, we can assume that $S\cap \mbox{Sing}(Q_{\Lambda})=\emptyset$. The linear projection $f_{\Lambda}:\Lambda \dashrightarrow \PP^3$ with center $\mbox{Sing}(Q_{\Lambda})$, induces a regular map $f:S\rightarrow \PP^3$ with image the smooth quadric $Q_0\subset \PP^3$.  Then $S$ is endowed with two elliptic pencils $|C_1|$ and $|C_2|$ corresponding to the projections of $Q_0\cong \PP^1\times \PP^1$ onto the two factors.  Since $C\in |\OO_S(1)|$, one has a linear equivalence
$C\equiv C_1+C_2$, on $S$.
As already pointed out, $\mbox{deg}(f)=C_1\cdot C_2=C^2/2=7$. The condition $\mbox{rk}(Q_{\Lambda}\cap \PP^7)=\mbox{rk}(Q_{\Lambda})-1$, implies that the hyperplane $\PP^7\in (\Lambda)^{\vee}$ is the pull-back of a hyperplane from $\PP^3$, that is, $\PP^7=f^{-1}_{\Lambda}(\Pi_0)$, where $\Pi_0\in (\PP^3)^{\vee}$.

We choose a general line $l_0\subset \Pi_0$ and denote by $\{q_1, q_2\}:=l_0\cap Q_0$. We consider the pencil $\{\Pi_t\}_{t\in \PP^1}\subset (\PP^3)^{\vee}$ of planes through $l_0$ as well as the induced pencil of curves of genus $8$
$$\{C_t:=f^{-1}(\Pi_t)\subset S\}_{t\in \PP^1},$$
each endowed with a vanishing theta-null induced by the
pencil $f_{t}: C_t\rightarrow Q_0\cap \Pi_t$.
\vskip 5pt

This pencil contains precisely two \emph{reducible} curves, corresponding to the planes $\Pi_{1}, \Pi_{2}$ in $\PP^3$ spanned by the rulings of $Q_0$ passing through $q_1$ and $q_2$ respectively. Precisely, if $l_i, m_i\subset Q_0$ are the rulings passing through $q_i$ such that $l_1\cdot l_2=m_1\cdot m_2=0$, then it follows that for $\Pi_1=\langle l_1, m_2\rangle, \Pi_2=\langle l_2, m_1\rangle$, the fibres $f^{-1}(\Pi_{1})$ and $f^{-1}(\Pi_2)$ split into two elliptic curves $f^{-1}(l_i)$ and $f^{-1}(m_j)$ meeting transversally in $7$ points. The half-canonical $\mathfrak g^1_7$ specializes to a degree $7$ admissible covering
 $$f^{-1}(l_i)\cup f^{-1}(m_j) \stackrel{f}\rightarrow l_i\cup m_j, \  \ i\neq j,$$ such that the $7$ points in $f^{-1}(l_i)\cap f^{-1}(m_j)$ map to $l_i \cap m_j$. To determine the point in $\ss_8^+$ corresponding to the admissible covering $\bigl(f^{-1}(l_i)\cup f^{-1}(m_j),\  f_{| f^{-1}(l_i)\cup f^{-1}(m_j)}\bigr)$, one must insert $7$ exceptional components at all the points of intersection of the two components. We denote by $R\subset \thet\subset \ss_8^+$ the pencil of spin curves obtained via this construction.

\begin{lemma}\label{nodalpencil} Each member  $C_t\subset S$ in the above constructed pencil is nodal. Moreover, each curve $C_t$ different from
$f^{-1}(l_1)\cup f^{-1}(m_2)$ and $f^{-1}(l_2)\cup f^{-1}(m_1)$ is irreducible. It follows that $R\cdot \alpha_i=R\cdot \beta_i=0$ for $i=1, \ldots, 4$.
\end{lemma}
\begin{proof} This follows since $f:S\rightarrow Q_0$ is a regular morphism and the base line $l_0\subset H_0$ of the pencil $\{\Pi_t\}_{t\in \PP^1}$ is chosen to be general.
\end{proof}

\begin{lemma} $R \cdot \pi^*(\mm^2_{7,8})= 0$.
\end{lemma}
\begin{proof} We show instead that $\pi_*(R)\cdot \mm_{8, 7}^2=0$. From Lemma \ref{nodalpencil}, the curve $R$ is disjoint  from  the divisors $A_i, B_i$ for $i=1, \ldots, 4$, hence  $\pi_*(R)$ has the numerical characteristics of a Lefschetz pencil of curves of genus $8$ on a fixed $K3$ surface.

\noindent
In particular, $\pi_*(R)\cdot \delta/\pi_*(R)\cdot \lambda=6+12/(g+1)=s(\mm_{8, 7}^2)$ and $\pi_*(R)\cdot \delta_i=0$ for $i=1, \ldots, 4$. This implies the statement. \end{proof}
\vskip 4pt

\begin{lemma}\label{thetintersection}
 $T \cdot \thet = -1$. \end{lemma}
 \begin{proof} We have already determined that
$
R\cdot \lambda =\pi_*(R)\cdot \lambda= \chi(\tilde{S}, \OO_{\tilde S}) + g - 1 = 9,
$ where $\tilde{S}:=\mbox{Bl}_{2g-2}(S)$ is the blow-up of $S$ at the points $f^{-1}(q_1)\cup f^{-1}(q_2)$.
Moreover,
\begin{equation}\label{relation}
R\cdot \alpha_0 + 2R\cdot \beta_0 =\pi_*(R)\cdot \delta_0= c_2(\tilde X) + 4(g-1) = 38 + 28 = 66.
\end{equation}
To determine $R\cdot \beta_0$ we study the local structure of $\ss_8^+$ in a neighbourhood of one of the two points, say $t^*\in R$ corresponding to a reducible curve, say $f^{-1}(l_1)\cup f^{-1}(m_2)$, the situation for $f^{-1}(l_2)\cup f^{-1}(m_1)$ being of course identical.
We set $\{p\}:=l_1\cap m_2\in Q_0$ and $\{x_1, \ldots, x_7\}:=f^{-1}(p)\in S$. We insert exceptional components $E_1, \ldots, E_7$ at the nodes $x_1, \ldots, x_7$ of $f^{-1}(l_1)\cup f^{-1}(m_2)$ and denote by $X$ the resulting quasi-stable curve. If $$\mu: f^{-1}(l_1)\cup f^{-1}(m_2)\cup E_1\cup \ldots \cup E_7\rightarrow f^{-1}(l_1)\cup f^{-1}(m_2)$$ is the stabilization morphism, we set $\{y_i, z_i\}:=\mu^{-1}(x_i)$, where $y_i\in E_i\cap f^{-1}(l_1)$ and $z_i\in E_i\cap f^{-1}(m_2)$ for $i=1, \ldots, 7$. If $t^*=[X, \eta, \beta]$, then $\eta_{f^{-1}(l_1)}=\OO_{f^{-1}(l_1)}, \ \eta_{f^{-1}(m_2)}=\OO_{f^{-1}(m_2)}$, and of course $\eta_{E_i}=\OO_{E_i}(1)$. Moreover, one computes that $\mbox{Aut}(X, \eta, \beta)=\mathbb Z_2$ and $\mbox{Aut}(f^{-1}(l_1)\cup f^{-1}(m_2))=\{\mbox{Id}\}$, cf. \cite{C} Lemma 2.2.

\vskip 5pt
If $\mathbb C_{\tau}^{3g-3}$  denote the versal deformation space of $[X, \eta, \beta]\in \ss_g^+$, then there are local parameters $(\tau_1, \ldots, \tau_{3g-3})$, such that for $i=1, \ldots, 7$, the locus $\bigl(\tau_i=0\bigr)\subset \mathbb C_{\tau}^{3g-3}$ parameterizes spin curves for which the exceptional component $E_i$ persists. It particular, the pull-back $\mathbb C_{\tau}^{3g-3}\times _{\ss_g^+} B_0$ of the boundary divisor $B_0\subset \ss_g^+$ is given by the equation $\bigl(\tau_1\cdots \tau_7=0\bigr)\subset \mathbb C_{\tau}^{3g-3}$. The group $\mbox{Aut}(X, \eta, \beta)$ acts on $\mathbb C_{\tau}^{3g-3}$ by
$$(\tau_1, \ldots, \tau_7, \tau_8, \ldots, \tau_{3g-3})\mapsto (-\tau_1, \ldots, -\tau_7, \tau_8, \ldots, \tau_{3g-3}),$$
and since an \'etale neighbourhood of $t^*\in \ss_g^+$ is isomorphic to $\mathbb C_{\tau}^{3g-3}/\mbox{Aut}(X, \eta, \beta)$, we find that $B_0$ is not Cartier around $t^*$ (though $2B_0$ is Cartier). It follows that the intersection multiplicity of $R\times _{\ss_g^+} \mathbb C_{\tau}^{3g-3}$ with the locus $(\tau_1\cdots \tau_7)=0$ equals $7$, that is, the intersection multiplicity of $R\cap \beta_0$ at the point $t^*$ equals $7/2$, hence
$$
R\cdot \beta_0 =\bigl(R\cdot \beta_0\bigr)_{f^{-1}(l_1)\cup f^{-1}(m_2)}+\bigl(R\cdot \beta_0\bigr)_{f^{-1}(l_2)\cap f^{-1}(m_1)}=\frac{7}{2}+\frac{7}{2}=7.
$$
Then using (\ref{relation}) we find that $R\cdot \beta_0=66-14=52$, and finally
$$
R\cdot \thet = \frac{1}{4}R\cdot \lambda  - \frac{1}{16}R\cdot \alpha_0 = \frac{9}{4}- \frac{52}{16}= -1.
$$
\end{proof}
\begin{remark}
The final argument in the previous proof, namely that the reducible curve $f^{-1}(l_1)\cup f^{-1}(m_2)$ contributes with multiplicity $7/2$ to  $R\cdot \beta_0$, can also be derived by interpreting $\thet$ as a space of admissible coverings of degree $7$ over the versal deformation space $\mathbb C_{\tau}^{3g-3}$ and then making a local analysis similar to the one in \cite{D} pg. 47-50.
\end{remark}

\section{The Kodaira dimension of $\mm_{11, 11}$}

We begin by recalling the notation for boundary divisor classes on the moduli space $\mm_{g, n}$.  For an integer $0\leq i\leq [g/2]$ and a set of labels
$T\subset \{1, \ldots, n\}$, we denote by $\Delta_{i: T}$ the closure in $\mm_{g, n}$ of the locus of $n$-pointed curves $[C_1\cup C_2, x_1, \ldots, x_n]$, where $C_1$ and $C_2$ are smooth curves of genera $i$ and $g-i$ respectively, and the marked points lying on $C_1$ are precisely those labeled
by $T$. As usual, we define $\delta_{i: T}:=[\Delta_{i: T}]\in \mbox{Pic}(\mm_{g, n})$. For $0\leq i\leq [g/2]$ and $0\leq c\leq g$, we set
$$\delta_{i: c}:=\sum_{\#(T)=c}\delta_{i: T}.$$
 By convention, $\delta_{0: c}:=\emptyset$, for $c<2$.
If $\phi:\mm_{g, n}\rightarrow \mm_g$ is the morphism forgetting the marked points, we set $\lambda:=\phi^*(\lambda)\in \mbox{Pic}(\mm_{g, n})$ and $\delta_{\mathrm{irr}}:=\phi^*(\delta_{\mathrm{irr}})\in \mbox{Pic}(\mm_{g, n})$, where $\delta_{\mathrm{irr}}:=[\Delta_{\mathrm{irr}}]\in \mbox{Pic}(\mm_g)$ denotes the class of the locus of irreducible nodal curves. Furthermore, $\psi_1, \ldots, \psi_n\in \mbox{Pic}(\mm_{g, n})$ are the cotangent classes corresponding to the marked points. The canonical class of $\mm_{g, n}$ has been computed, cf. \cite{Log} Theorem 2.6:
\begin{equation}\label{canmgn}
K_{\mm_{g, n}}\equiv 13\lambda-2\delta_{\mathrm{irr}}+\sum_{i=1}^n \psi_i-2\sum_{i\geq 0, T} \delta_{i: T}-\sum_{T} \delta_{1: T}.
\end{equation}
We show that, at least for small $g$, the divisor $\dd_g$ of curves with $g$ marked points moving in a pencil, is an extremal point in the effective cone of $\mm_{g, g}$:
\begin{proposition}\label{extrem}
For $3\leq g\leq 11$, the irreducible divisor $\dd_g$ is filled up by  rational curves $R\subset \mm_{g, g}$ such that $R\cdot \dd_g<0$. It follows that $[\dd_g]\in \mathrm{Eff}(\mm_{g, g})$ is a rigid divisor. Moreover, when $g\neq 10$, one can assume that $R\cdot \delta_{i: T}=0$ for all $i\geq 0$ and $T\subset \{1, \ldots, g\}$.
\end{proposition}
\begin{proof}
We first treat the case $g\neq 10$, and start with a general point $[C, x_1, \ldots, x_g]\in \cD_g$. We assume that the points $x_1, \ldots, x_g\in C$ are distinct and $h^0(C, K_C(-x_1-\cdots -x_g))=1$. Let us consider the $(g-2)$-dimensional linear space $$\Lambda:=\langle x_1, \ldots, x_g\rangle\subset \PP\bigl(H^0(C, K_C)^{\vee}\bigr)=\PP^{g-1}.$$ Since $\phi(\cD_g)=\cM_g$, we may assume that $[C]\in \cM_g$ is a general curve. In particular, $C$ lies on a $K3$ surface $S\stackrel{|\OO_S(C)|}\hookrightarrow \PP^g$, which admits the canonical curve $C$ as a hyperplane section, cf. \cite{M1}. We intersect $S$ with the pencil of hyperplanes $\{H_{\lambda}\in (\PP^g)^{\vee}\}_{\lambda\in \PP^1}$ such that $\Lambda\subset H_{\lambda}$. Since (i) the locus of hyperplanes  $H\in (\PP^g)^{\vee}$ such that the intersection $S\cap H$ is not nodal has codimension $2$ in $(\PP^g)^{\vee}$, \ and (ii) the pencil $\{H_{\lambda}\}_{\lambda\in \PP^1}$ can be viewed as a general pencil of hyperplanes containing $\PP\bigl(H^0(C, K_C)^{\vee}\bigr)$ as a member, we may assume that all the curves $H_{\lambda}\cap S$ are nodal and that the nodes stay away from the fixed points $x_1, \ldots, x_g$. In this way we obtain a family in $\mm_{g, g}$
$$R:=\{[C_{\lambda}:=H_{\lambda}\cap S,\  x_1, \ldots, x_g]:\Lambda\subset H_{\lambda}, \ \lambda \in \PP^1\},$$
inducing a fibration $f:\tilde{S}:=\mbox{Bl}_{2g-2}(S)\rightarrow \PP^1$, obtained by blowing-up the base points of the pencil, together with $g$ sections given by the exceptional divisors
$E_{x_i}\subset \tilde{S}$ corresponding to the base points $x_1, \ldots, x_g$. The numerical parameters of $R$ are computed using, for instance, \cite{FP} Section 2. Precisely, one writes that
\begin{equation}\label{numericalparameters}
R\cdot \lambda=(\phi_*(R)\cdot \lambda)_{\mm_g}=g+1,\ \  R\cdot \delta_{\mathrm{irr}}=(\phi_*(R)\cdot \delta_{\mathrm{irr}})_{\mm_g}=6g+18, \ \ R\cdot \delta_{i: T}=0,
\end{equation}
for $i\geq 0$ and $T\subset \{1, \ldots, g\}$. Finally, from the adjunction formula, $R\cdot \psi_i=-(E_{x_i}^2)_{\tilde{S}}=1$ for $1\leq i\leq g$. Thus, $R\cdot \dd_g=-1$. Since $R$ is a covering
curve for the divisor $\dd_g$, it follows that $\dd_g$ is a rigid divisor on $\mm_{g, g}$.
\vskip 5pt

We turn to the case $g=10$, when the previous argument breaks down because the general curve $[C]\in \cM_{10}$ no longer lies on a $K3$ surface.
More generally, we fix $g<11, g\neq 9$ and pick a general point $[C, x_1, \ldots, x_{g}]\in \cD_{g}$. We denote by $X:=C_{ij}$ the nodal curve obtained from $C$ by identifying $x_i$ and $x_j$, where $1\leq i<j\leq g$. Since $[X]\in \Delta_0\subset \mm_{g+1}$ is a general $1$-nodal curve of genus $g+1$, using e.g. \cite{FKPS}, there exists a smooth $K3$ surface $S$ containing $X$.  We denote by $\nu:C\rightarrow X\subset S$ the normalization map and set $\nu(x_i)=\nu(x_j)=p$. The linear system $|\OO_S(X)|$ embeds $S$ in $\PP^{g+1}$ and $\nu^*(\OO_S(X))=K_C(x_i+x_j)$. Let $\epsilon:S':=\mbox{Bl}_p(S)\rightarrow S$ be the blow-up of $S$ at $p$ and $E\subset S'$ the exceptional divisor. Note that $C$ viewed as an embedded curve in $S'$ belongs to the linear system $|\epsilon^*\OO_S(1)\otimes \OO_{S'}(-2E)|$ and $C\cdot E=x_i+x_j$. Let $Z\subset S'$ the reduced $0$-dimensional scheme consisting of marked points of $C$ with support $\{x_i, x_j\}^c$.

Since $h^0(C, \OO_C(x_1+\cdots+x_{g}))=2$, we find  that $Z$ together with the tangent plane
$\mathbb{T}_p(X)=\mathbb{T}_p(S)$ span a $(g-1)$-dimensional linear space $\Lambda\subset \PP^{g+1}$. We obtain a $1$-dimensional family in $\dd_g$ by taking the normalization of the intersection curves on $S$ with hyperplanes $H\in (\PP^{g+1})^{\vee}$ passing through $\Lambda$. Equivalently, we note that $$h^0(S', \mathcal{I}_{Z/S'}(C))=h^0(S', \OO_{S'})+h^0(C,K_C(-x_1-\cdots -x_{g}))=2,$$ that is, $|\mathcal{I}_{Z/S'}(C)|$  is a pencil of curves on $S'$. We denote by $\tilde{\epsilon}:\tilde{S}:=\mbox{Bl}_{2g-4}(S')\rightarrow S'$ the blow-up of $S'$ at the $(\epsilon^*(H)-2E)^2=2g-4$ base points of $|\mathcal{I}_{Z/S'}(C)|$, by $f:\tilde{S}\rightarrow \PP^1$ the induced fibration with $(g-2)$ sections
corresponding to the points of $Z$, as well as with a $2$-section given by the divisor $E:=\tilde{\epsilon}^{-1}(E)$. Since $\mbox{deg}(f_{E})=2$, there are precisely two fibres of $f$, say $C_1$ and $C_2$, which are tangent to $E$.  We make a base change or order $2$ via the morphism $f_{E}:E\rightarrow \PP^1$, and consider the fibration
 $$q':Y':=\tilde{S}\times_{\PP^1} E\rightarrow E.$$
Thus $p:Y'\rightarrow \tilde{S}$ is the double cover branched along $C_1+C_2$. Clearly $q'$ admits two sections $E_1, E_2\subset Y'$ such that $p^*(E)=E_1+E_2$ and $E_1\cdot E_2=2$. By direct calculation, it follows that $E_1^2=E_2^2=-3$. To separate the sections
 $E_1$ and $E_2$, we blow-up the two points of intersection $E_1\cap E_2$ and we denote by $q:Y:=\mathrm{Bl}_2(Y')\rightarrow E$  the resulting fibration, which possesses everywhere distinct sections $\sigma_i:E\rightarrow Y'$ for $1\leq i\leq g$, given by the proper transforms of $E_1$ and $E_2$ as well as  the proper transforms of the exceptional divisors corresponding to the points in $Z$.
The numerical characters of the family $\Gamma_{ij}:=\{[q^{-1}(t), \sigma_1(t), \ldots, \sigma_g(t)]: t\in E\} \subset \mm_{g, g}$ are computed as follows:
$$\Gamma_{ij}\cdot \lambda=2(g+1),\ \mbox{ }  \Gamma_{ij}\cdot \delta_{\mathrm{irr}}=2(6g+17), \  \Gamma_{ij}\cdot \psi_l=2 \mbox{ for }  l\in \{i, j\}^c,$$
$$ \Gamma_{ij}\cdot \psi_i=\Gamma_{ij}\cdot \psi_j=-(E_i^2)_{Y'}+2=5, \ \Gamma_{ij}\cdot \delta_{0: ij}=2,  \mbox{ } \Gamma_{ij}\cdot \delta_{l: T}=0 \mbox{ for } l\geq 0,  T\subset \{i, j\}^c.$$
We take the $\mathfrak S_g$-orbit of the $1$-cycle $\Gamma_{ij}$ with respect to permuting the marked points,
$$\Gamma:=\frac{1}{g(g-1)} \sum_{i<j} \Gamma_{ij}\in NE_1(\mm_{g, g}),$$
and note that $\Gamma\cdot \dd_{g}=-1$. Each component $\Gamma_{ij}$ fills-up $\dd_g$, which finishes the proof.
\end{proof}
\vskip 3pt
%A variation of the idea described in the proof above, yields the following result:
%\begin{theorem}
%The moduli space $\mm_{g, 10}$ is uniruled for $g\leq 11, g\neq 10$. In particular, $\zeta(g)\geq 11$ in this range.
%\end{theorem}
%\begin{proof} Let $\cM$ be the moduli space of polarized $K3$ surface $(S, h)$ of degree $2g-2$, such that $\mathrm{Pic}(S)$ contains classes of %smooth curves $C\in |h|$ as well as  $N_1, \ldots, N_{11-g}$, where $C^2=2g-2, \ N_i^2=-2, \ N_i\cdot C=1$ for $1\leq i\leq 11-g$ and $N_i\cdot N_j=0$ %for $i\neq j$. If $\beta:\cU \rightarrow \cM$ denotes the universal $K3$ surface, then there exists a $\PP^1$-bundle $\P$ defined over an open subset %of the $(g-1)$-fold product $\cU \times_{\cM} \cdots \times_{\cM} \cU$ \ classifying hyperplanes \ $H\in (\PP^g)^{\vee}=\PP H^0(S, \OO_S(C))^{\vee}$ %containing the points $x_1,  \ldots, x_{g-1}\in S$. We observe that $\mbox{dim}(\P)=\mbox{dim}(\cM_{g, 10})$. The conclusion follows by considering %the \emph{dominant} rational map $\P \dashrightarrow \mm_{g, 10}$, given by
%$$\bigl((S, x_1, \ldots, x_{g-1}), H \bigr) \mapsto \bigl[C_H:=H\cap S,\  x_1, \ldots, x_{g-1}, \ C_H\cdot N_1, \ldots, C_H\cdot N_{11-g}\bigr]\in %\mm_{g, 10}.$$
%\end{proof}
%\begin{remark} Note that $\kappa(\mm_{10, 10})\geq 0$, cf. \cite{FP} Proposition 7.5.
%\end{remark}

\vskip 4pt

We now specialize to the case of genus $11$: On $\mm_{11}$ there exist two divisors of Brill-Noether type,  namely the closure of the locus of $6$-gonal curves   $$\cM_{11, 6}^1:=\{[C]\in \cM_{11}: G^1_6(C)\neq \emptyset\}$$
and the closure of the locus $\cM_{11, 9}^2:=\{[C]\in \cM_{11}: G^2_9(C)\neq \emptyset\}$.  The divisors $\mm_{11, 6}^1$ and $\mm_{11, 9}^2$ are irreducible, distinct, and their classes  are proportional, cf. \cite{EH2}. Precisely, there  are explicit constants $c_{11, 6}^1, c_{11, 9}^2\in \mathbb Z_{>0}$, such that $$\mathfrak{bn}_{11}:\equiv \frac{1}{c_{11, 6}^1}\ \mm_{11, 6}^1\equiv \frac{1}{c_{11, 9}^2}\ \mm_{11, 9}^2\equiv 7\lambda-\delta_0-5\delta_1-9\delta_2-12\delta_3-14\delta_4-15\delta_5\in \mbox{Pic}(\mm_{11}).$$
By interpolating, we find the following explicit canonical divisor:
\begin{equation}\label{canrep}
K_{\mm_{11, 11}}\equiv \dd_{11}+ 2\cdot  \phi^*(\mathfrak{bn}_{11})+\sum_{i=0}^5\sum_{c=0}^{11} d_{i: c}\ \delta_{i: c},
\end{equation}
where $$d_{0: c}=\frac{c^2+c-4}{2}\ \ \mbox{ for }c\geq 2,\  \ d_{1: c}=7+{|c-1|+1\choose 2}, \ \  d_{2: c}=16+{|c-2|+1\choose 2},$$
$$d_{3: c}=22+{|c-3|+1\choose 2}, \ \  d_{4: c}=26+{|c-4|+1\choose 2}, \ \  d_{5: c}=28+{|c-5|+1\choose 2}.$$
One already knows that multiples of $\dd_{11}$ are non-moving divisors on $\mm_{11, 11}$. We show that $\dd_{11}$ does not move in any multiple of the canonical linear system on $\mm_{11, 11}$.
\begin{proposition}
For each integer $n\geq 1$ one has an isomorphism
$$H^0\bigl(\mm_{11, 11}, \OO_{\mm_{11, 11}}(nK_{\mm_{11, 11}})\bigr)\cong H^0\bigl(\mm_{11, 11}, \OO_{\mm_{11, 11}}(nK_{\mm_{11, 11}}-n\dd_{11})\bigr).$$ In particular, $\kappa\bigl(\mm_{11, 11}\bigr)=\kappa\bigl(\mm_{11, 11}, K_{\mm_{11, 11}}-\dd_{11}\bigr)$.
\end{proposition}
\begin{proof} Using the notation and results from Proposition \ref{extrem}, we recall that we have constructed a curve $R\subset \mm_{11, 11}$ moving in a family which  fills-up the divisor $\dd_{11}$, such that  $R\cdot \dd_{11}=-1$ and $R\cdot \delta_{i: S}=0$, for all $i\geq 0$ and $T\subset \{1, \ldots, g\}$.
All points in $R$ correspond to nodal curves lying  on a fixed $K3$ surface $S$, which by the generality assumptions, can be chosen such that
$\mbox{Pic}(S)=\mathbb Z$. Applying \cite{Laz}, all underlying genus $11$ curves corresponding to points in $R$ satisfy the Brill-Noether theorem, in particular $R\cdot \phi^*(\mathfrak{bn}_{11})=0$, that is, $R\cdot K_{\mm_{11, 11}}=R\cdot \dd_{11}=-1$. It follows that for any effective divisor $E$ on $\mm_{11, 11}$ such that $E\equiv nK_{\mm_{11, 11}}$, one has that $R\cdot E=-n$, thus the class $E-n\dd_{11}$ is still effective and then $|nK_{\mm_{11, 11}}|=n\dd_{11}+|nK_{\mm_{11, 11}}-n\dd_{11}|$.
\end{proof}
We are in a position to complete the proof of Theorem \ref{gen11}:
\begin{theorem} We have that $\kappa\bigl(\mm_{11, 11}, \ 2\cdot \phi^*(\mathfrak{bn}_{11})+\sum_{i, c} d_{i: c}\cdot\delta_{i: c}\bigr)=19$.
It follows that the Kodaira dimension of $\mm_{11, 11}$ equals $19$.
\end{theorem}
\begin{proof} To simplify the proof, we define a few divisors classes on $\mm_{11, 11}$:
 $$A:=2\cdot \phi^*(\mathfrak{bn}_{11})+\sum_{i\geq 0, c} d_{i: c}\ \delta_{i: c}\equiv K_{\mm_{11, 11}}-\dd_{11}\ \mbox{ and } \ A':=A-\sum_{c=2}^{11} d_{0: c}\
 \delta_{0: c},$$
 as well as, $B:= \mathfrak{bn}_{11}+4\delta_3+7\delta_4+8\delta_5\in \mbox{Pic}(\mm_{11})$.

We claim that for all integers $n\geq 1$ one has isomorphisms,
$$H^0\bigl(\mm_{11, 11}, \OO_{\mm_{11, 11}}(nA)\bigr)\cong H^0\bigl(\mm_{11, 11}, \OO_{\mm_{11, 11}}(nA')\bigr).$$
Indeed, we fix a set of labels $T\subset \{1, \ldots, 11 \}$ such that $\#(T)\geq 2$ and consider a pencil
$$\bigl\{[C_t, x_i(t), p(t): i\in T^c]\bigr\}_{t\in \PP^1} \subset \mm_{11, 12-\#(T)}, $$ of $(12-\#(T))$-pointed curves of genus $11$ on a general $K3$ surface $S$, with marked points being labeled by elements in $T^c$ as well by another label $p(t)$. The pencil is induced by a fibration obtained from a Lefschetz pencil of genus $11$ curves on $S$, with regular sections given by $(12-\#(T))$ of the exceptional divisors obtained by blowing-up $S$ at the $(2g-2)$ base points of the pencil. To each element in this pencil, we attach at the marked point labeled by $p(t)$, a fixed copy of $\PP^1$ together with fixed marked points $x_i\in \PP^1-\{\infty\}$, for $i\in T$. The gluing identifies the point $p(t)\in C_t$ with $\infty\in \PP^1$. If $R_T\subset \mm_{11, 11}$ denotes the resulting family,  we compute:
$$R_T\cdot \lambda=g+1,\ R_T\cdot \delta_{\mathrm{irr}}=6(g+3), \ \ R_T\cdot \delta_{0: T}=-1,\ R_T\cdot \psi_i=1 \mbox{ for }i\in T^c, \  \ R_T\cdot \psi_i=0\mbox{ for }i\in T.$$ Moreover, $R_T$ is disjoint from all remaining boundary divisors of $\mm_{11, 11}$. One finds that
$R_T\cdot \phi^*(\mathfrak{bn}_{11})=0$.  Thus for any effective divisor $E\subset \mm_{11, 11}$ such that $E\equiv nA$, we find that $R_T\cdot E=-n d_{0, c}$.

Since for all $T$, the pencil $R_T$  fills-up the divisor $\Delta_{0: T}$, we can deform the curves $R_T\subset \Delta_{0: T}$, to find that $E-\sum_{c=2}^{11}nd_{0: c} \cdot \delta_{0: c}$
is still an effective class, that is,
$$|nA|=\sum_{c=2}^{11} nd_{0: c}\cdot \Delta_{0: c}+|nA'|,$$
which proves the claim.
Next, by direct calculation we observe that the class $A'-2\phi^*(B)$ is effective.  Zariski's Main Theorem gives that  $\phi_*\phi^*\OO_{\mm_{11}}(B)=\OO_{\mm_{11}}(B)$, thus
$$\kappa\bigl(\mm_{11, 11}, A'\bigr)\geq \kappa\bigl(\mm_{11, 11}, \phi^*(B)\bigr)=\kappa(\mm_{11}, B)=19.$$
The last equality comes from \cite{FP} Proposition 6.2: The class $B$ contains the pull-back of an ample class under the Mukai map \cite{M3} $$q_{11}: \mm_{11, 11}\dashrightarrow \ff_{11},\ \  \ [C, x_1, \ldots, x_{11}]\mapsto [S\supset C,\  \OO_S(C)],$$
to a compactification of the moduli space of polarized $K3$ surfaces of degree $20$.

On the other hand, since $\phi^*(\delta_i)=\sum_{S} \delta_{i: S}$ for $1\leq i\leq 5$, there is a divisor class on $\mm_{11}$ of type
$B':= 2\cdot \mathfrak{bn}_{11}+\sum_{i=1}^5 a_i \delta_i\in \mathrm{Pic}(\mm_{11})$,
with $a_i\geq 0$, such that $\phi^*(B')-A'$ is an effective divisor. It follows that
$$\kappa\bigl(\mm_{11, 11}, A'\bigr)\leq \kappa\bigl(\mm_{11, 11}, \phi^*(B')\bigr)=\kappa(\mm_{11}, B').$$
If $R_{11} \subset \mm_{11}$ is the family corresponding to a Lefschetz pencil of curves of genus $11$ on a fixed $K3$ surface, then $R_{11}\cdot B'=0$. The pencil $R_{11}$ moves in a $11$-dimensional family inside $\mm_{11}$ which is contracted to a point by any linear series $|nB'|$ on $\mm_{11}$ with $n\geq 1$ (in fact a general curve $R_{11}$ is \emph{disjoint} from the base locus of $|nB'|$). One
finds that $\kappa(\mm_{11}, B')\leq 19$, which completes the proof.
\end{proof}
\section{The uniruledness of $\mm_{g, n}$}

We formulate a general principle, somewhat similar to the one used in the proof of Theorem \ref{spin8},  which we use in proving the uniruledness of some moduli spaces $\mm_{g, n}$. The next result, although simple, can be applied to determine \emph{all} uniruled moduli spaces $\mm_{g, n}$ for $g\leq 8$:
\begin{proposition}\label{uniruled}
Let $X$ be a projective $\mathbb Q$-factorial variety and suppose $D_1, D_2\subset X$ are irreducible effective $\mathbb Q$-divisors such that there exist covering curves
$\Gamma_i\subset D_i$, with $\Gamma_i\cdot D_i<0$ for $i=1, 2$ (in particular both $D_i\in \mathrm{Eff}(X)$ are non-movable divisors). Assume furthermore that
\begin{equation}
\begin{vmatrix}
\Gamma_1\cdot D_1& \Gamma_1\cdot D_2\\
\Gamma_2\cdot D_1& \Gamma_2\cdot D_2\\
\end{vmatrix} \leq 0,\mbox{ }\mbox{ }\
\begin{vmatrix}
\Gamma_1\cdot K_X& \Gamma_1\cdot D_1\\
\Gamma_2\cdot K_X& \Gamma_2\cdot D_1\\
\end{vmatrix}<0.
\end{equation}
Then $X$ is a uniruled variety.
\end{proposition}
\begin{proof} According to \cite{BDPP}, it suffices to prove that $K_X$ is not pseudo-effective. By contradiction, we choose $\alpha, \beta\in \mathbb R_{\geq 0}$ maximal such that $K_X-\alpha D_1-\beta D_2\in \overline{\mathrm{Eff}}(X)$. Then we can write down the inequalities
$$\Gamma_1 \cdot K_X\geq \alpha (\Gamma_1\cdot D_1) +\beta (\Gamma_1\cdot D_2) \ \mbox{ and } \Gamma_2\cdot K_X\geq \alpha (\Gamma_2\cdot D_1)+\beta (\Gamma_2\cdot  D_2).$$
Eliminating $\alpha$, the resulting inequality contradicts the assumption $\beta\geq 0$.
\end{proof}
We turn our attention to the proof of Theorem \ref{genul8}, which we split in three parts:

\begin{theorem}\label{gen5} $\mm_{5, n}$ is uniruled for $n\leq 14$.
\end{theorem}
\begin{proof}
A general $2$-pointed curve $[C, x, y]\in \cM_{5, 2}$ carries a finite number of linear series $L\in W^2_6(C)$, such that
if $\nu_L: C\stackrel{|L|}\longrightarrow \Gamma\subset \PP^2$ is the induced plane model, then $\nu_L(x)=\nu_L(y)=p_1$. Note that
$\Gamma$ has nodes, say $p_1, \ldots, p_{5}$, and $\mbox{dim }|\OO_{\PP^2}(\Gamma)(-2\sum_{i=1}^5 p_i)|=12$.

We pick general points $\{x_i\}_{i=1}^{11}, \ \{p_j\}_{j=1}^5 \subset \PP^2$, and a general line $l\subset \PP^2$. One considers the pencil of sextics passing simply through $x_1, \ldots, x_{11}$ and having nodes (only) at $p_1, \ldots, p_5$. The pencil induces a fibration $f':S\rightarrow \PP^1$, where $S:=\mathrm{Bl}_{21}(\PP^2)$ is obtained from $\PP^2$ by blowing-up $p_1, \ldots, p_5$, $x_1, \ldots, x_{11}$, as well as the remaining unassigned base points of the pencil. The exceptional divisors $E_{x_i}\subset S$ provide $11$ sections of $f'$. The exceptional divisor $E_{p_1}$ induces a $2$-section. Making a base change via the map $f'_{E_{p_1}}: E_{p_1}\rightarrow \PP^1$,  the $2$-section $E_{p_1}$ splits into two sections $E_x$ and $E_y$  meeting at $2$ points. Blowing these points up, we arrive at a
fibration $f:Y\rightarrow E_{p_1}$, carrying $13$ everywhere disjoint sections, $\tilde{E}_x, \tilde{E}_y, \tilde{E}_{x_1}, \ldots, \tilde{E}_{x_{11}}$, where $\tilde{E}_{x_i}\subset Y$ denotes the inverse image of $E_{x_i}$, and $\tilde{E}_x, \tilde{E}_y$ denote the proper transforms of $E_x$ and $E_y$ respectively. This induces a family of pointed stable curves
$$\Gamma:=\bigl\{[C_{\lambda}:=f^{-1}(\lambda),\  \tilde{E}_x\cdot C_{\lambda},\  \tilde{E}_y\cdot C_{\lambda},\  \ \tilde{E}_{x_1}\cdot C_{\lambda}, \ldots, \ \tilde{E}_{x_{11}}\cdot C_{\lambda}]:
\lambda\in E_{p_1}\bigr\}\subset \mm_{5, 13}.$$

We compute the numerical characters of $\Gamma$ (see also the proof of Proposition \ref{extrem}):
$$\Gamma\cdot \lambda=\mathrm{deg}(f_{E_{p_1}})\bigl(\chi(S, \OO_S)+g-1\bigr)=10,\ \Gamma\cdot \delta_{\mathrm{irr}}=\mathrm{deg}(f_{E_{p_1}})\bigl(c_2(S)+4g-4\bigr)=80,$$
$$\Gamma\cdot \psi_x=\Gamma\cdot \psi_y=5, \ \Gamma\cdot \psi_{x_1}=\cdots =\Gamma\cdot \psi_{x_{11}}=2, \ \Gamma\cdot \delta_{0: xy}=2,$$
whereas $\Gamma$ is disjoint from the remaining boundary divisors. Furthermore, $f$ admits a $6$-section given by the points of intersection
$l\cdot C_{\lambda}$. Making a base change of order $6$ via $f_{l}:l\rightarrow E_{p_1}$, we obtain a family $U\subset \mm_{5, 14}$, with
 marked points  labeled by $\bigl(x(t), y(t), x_1(t),\ldots, x_{11}(t), l(t)\bigr)_{t\in l}$.
  Clearly $(U\cdot \alpha)_{\mm_{5, 14}}=\mbox{deg}(f_{l})\bigl(\Gamma\cdot \alpha\bigr)_{\mm_{5, 13}}$, for every class $\alpha\in \{\lambda, \psi_x, \ \psi_y,\  \psi_{x_1},\ldots, \psi_{x_{11}},\   \delta_{\mathrm{irr}}, \delta_{0: xy} \}\subset \mbox{Pic}(\mm_{5, 14})$. Finally, one has that
$$U\cdot \psi_l=-(l^2)_{Y}+\bigl(2\mathrm{deg}(f_{l})-2\bigr)=10.$$
By direct calculation one finds, $U\cdot K_{\mm_{5, 14}}=-2$, which completes the proof.
\end{proof}
\begin{remark} It is known, cf. \cite{F2} Section 5,  that $\mm_{5, 15}$ is of general type. Thus Theorem \ref{gen5} settles completely the classification problem for $\mm_{5, n}$. In particular, $\zeta(5)=15$.
\end{remark}

\begin{theorem}\label{gen78} $\mm_{8, n}$ is uniruled for $n\leq 12$.
\end{theorem}
\begin{proof} We apply Proposition \ref{uniruled} for $D_1:=\mathfrak{bn}_8$ (see Section 2), and   $D_2:=\Delta_{\mathrm{irr}}\in \mathrm{Eff}(\mm_{8, n})$. To construct a covering curve $\Gamma_1\subset D_1$, we lift to $\mm_{8, n}$ a Lefschetz pencil of $7$-nodal plane septics.  The fibration $f:\mathrm{Bl}_{28}(\PP^2)\rightarrow \PP^1$ constructed in the course of proving Proposition \ref{septics}, carries $n$ sections given by the exceptional divisors corresponding to $n$ unassigned base points. If $\Gamma_1\subset \mm_{8, n}$ denotes the resulting, then
$$\Gamma_1\cdot \lambda=\phi_*(\Gamma_1) \cdot \lambda=8, \ \Gamma_1\cdot \delta_{\mathrm{irr}}=\phi_*(\Gamma_1) \cdot \delta_{\mathrm{irr}}=59,  \ \Gamma_1\cdot \psi_i=1 \mbox{ for } i=1, \ldots, n,$$
and $\Gamma_1 \cdot \delta_{i: T}=0$. It follows that $\Gamma_1\cdot D_1=-1/3, \ \Gamma_1\cdot K_{\mm_{8, n}}=n-14$ and $\Gamma_1\cdot D_2=59$.

We construct a covering curve $\Gamma_2\subset D_2$ and start with a general pointed curve $[C, x_1, \ldots, x_{n+1}]\in \mm_{7, n+1}$. We identify $x_{n+1}$ with a moving point $y\in C$, that is, take
$$\Gamma_2:=\bigl\{\bigl[\frac{C}{y\sim x_{n+1}},  x_1, \ldots, x_{n} \bigr]: y\in C\bigr\}\subset \mm_{8, n}.$$
It is easy to compute that
$\Gamma_2\cdot \lambda=0$, \  $\Gamma_2\cdot \delta_{\mathrm{irr}}=-2g(C)=-14$, \ $\Gamma_2\cdot \delta_{1: \emptyset}=1,\ \Gamma_2\cdot \psi_i=1,
\mbox{ for } i=1, \ldots, n,$
and $\Gamma_2\cdot \delta_{i: T}=0$ for $(i, T)\neq (1, \emptyset)$. Therefore $\Gamma_2\cdot D_1=28/3$ and $\Gamma_2\cdot K_{\mm_{8, n}}=25+n$. The conditions of Proposition \ref{uniruled} are satisfied for $n\leq 12$.
\end{proof}
\begin{remark} The results of Theorem \ref{gen78} are almost optimal. The space $\mm_{8, 14}$ is of general type, cf. \cite{Log}. The Kodaira dimension of $\mm_{8, 13}$ is still unknown. Note that it was already known \cite{Log}, \cite{CF}, that
$\mm_{8, n}$ is unirational for $n\leq 11$.
\end{remark}
\vskip 3pt

Finally, we turn to the case of genus $7$.
In order to establish the uniruledness of $\mm_{7, n}$, we consider the following  effective divisors on $\mm_{7, n}$:
$$\cD_1:=\{[C, x_1, \ldots, x_n]\in \cM_{7, n}: \exists L\in W^2_7(C)\mbox{  with } h^0(C, L(-x_1-x_2))\geq 1\}, $$
and $\dd_2:= \phi^*(\mathfrak{bn}_{7})$, where $\mathfrak{bn}_7:=\frac{1}{c_{7, 4}^1}\mm_{7, 4}^1\equiv 15\lambda-2\delta_0-9\delta_1-15\delta_2-18\delta_3 \in \mathrm{Pic}(\mm_7)$
is the linear system spanned by the unique Brill-Noether divisor on $\mm_7$. Before computing the class $[\dd_1]$,  we need a calculation, which may be of independent interest:

\begin{proposition}\label{divm72}
Let $g\equiv 1  \ \mathrm{ mod } \ 3$ be a fixed integer and set $d:=(2g+7)/3$, so that the Brill-Noether number $\rho(g, 2, d)=1$. One considers the effective divisor of nodes of plane curves
$$\mathfrak{Node}_g:=\{[C, x, y]\in \cM_{g, 2}: \exists L\in W^2_d(C) \mbox{ such that } \ h^0(C, L(-x-y))\geq 2\}.$$
The class of its closure in $\mm_{g, 2}$ is given by the formula:
$$\overline{\mathfrak{Node}}_g\equiv c_g\Bigl((g+4)\lambda+\frac{g+2}{6}(\psi_1+\psi_2)-\frac{g+2}{6}\delta_{\mathrm{irr}}-g\delta_{0: 12}-\cdots\ \Bigr)\in \mathrm{Pic}(\mm_{g, 2}),$$
$$\mbox{ where }\ \  c_g:=\frac{24(g-2)!}{(g-d+5)!\ (g-d+3)!\ (g-d+1)!}.$$
\end{proposition}
\begin{proof} We denote by $\phi_1:\mm_{g, 2}\rightarrow \mm_{g, 1}$ the morphism forgetting the second marked point. The divisor $\overline{\mathfrak{Cu}}_g:=(\phi_1)_*(\overline{\mathfrak{Node}}_g\cdot \delta_{0: 12})$ coincides with the cusp locus in $\mm_{g, 1}$, that is, the closure of
the locus of pointed curves $[C, x]\in \cM_{g, 1}$, such that there exists $L\in W^2_d(C)$ with $h^0(C, L(-2x))\geq 2$.

In order to compute its class, we fix a general elliptic curve $[E, x]\in \mm_{1, 1}$ and consider the map $j:\mm_{g, 1}\rightarrow \mm_{g+1}$,  given by $j([C, x]):=[C\cup_x
E]$. Then $$\overline{\mathfrak{Cu}}_g=j^*(\mm_{g+1, d}^2),$$ where $\mm_{g+1, d}^2$ is the Brill-Noether divisor on $\mm_{g+1}$ consisting of curves with a
$\mathfrak g^2_d$ (Note that $\rho(g+1, 2, d)=-1$). Since the class $[\mm_{g+1, d}^2]\in \mathrm{Pic}(\mm_{g+1})$ is known, cf. \cite{EH2}, and  $j^*(\lambda)=\lambda$, \ $j^*(\delta_{\mathrm{irr}})=\delta_{\mathrm{irr}}$, $j^*(\delta_1)=-\psi+\delta_{g-1: 1}$, we obtain the following expression
$$\overline{\mathfrak{Cu}}_g\equiv c_g \Bigl((g+4)\lambda+g\psi-\frac{g+2}{6}\delta_{\mathrm{irr}}-\sum_{i=1}^{g-1} (i+1)(g-i)\delta_{i: 1} \Bigr)\in \mathrm{Pic}(\mm_{g, 1}).$$
Using the obvious formulas $(\phi_1)_*(\lambda\cdot \delta_{0: 12})=\lambda$,\ $(\phi_1)_*(\delta_{0: 12}^2)=-\psi$, $(\phi_1)_*(\delta_{\mathrm{irr}}\cdot \delta_{0: 12})=\delta_{\mathrm{irr}}$ and $(\phi_1)_*(\psi_i\cdot \delta_{0: 12})=0$ for $i=1, 2$, one finds that the
$\delta_{0:12}$-coefficient of $\overline{\mathfrak{Node}}_g$ equals the $\psi_1$-coefficient of $\overline{\mathfrak{Cu}}_g$, while the
$\lambda, \delta_{\mathrm{irr}}$-coefficients coincide.

We determine the $\psi_1$-coefficient in $[\overline{\mathfrak{Node}}_g]$. To this end, we fix a general point $[C, q]\in
\cM_{g, 1}$ and consider the test curve $C_2:=\{[C, q, y]: y\in C\}\subset \mm_{g, 2}$. Then, $C_2\cdot \psi_1=1$, $C_2\cdot \psi_2=2g-1$ and obviously $C_2\cdot \delta_{0: 12}=1$. On the other hand, $C_2\cdot \overline{\mathfrak{Node}}_g$ equals the number of points $y\in C$, such that for some (necessarily complete and base point free) linear series $L\in W^2_d(C)$, the morphism $$\chi(y, L): L_{|y+q}^{\vee}\rightarrow H^0(C, L)^{\vee} $$ fails to be injective. The map $\chi(y, L)$ globalizes to a morphism $\chi$ between vector bundles over $C\times W^2_d(C)$, and the number in question is expressed as the Chern number of the top degeneracy locus of $\chi$. Precisely, if $\P$ denotes a Poincar\'e bundle on $C\times \mbox{Pic}^d(C)$ and $\nu: C\times W^2_d(C)\rightarrow W^2_d(C)$ is the second projection, then one has that
$$C_2\cdot \overline{\mathfrak{Node}}_g= -2\theta\cdot [W^2_d(C)]+(d-1)c_1 \bigl(\nu_* \P^{\vee}\big),$$
where the Chern number $c_1(\nu_*\P^{\vee})$ can be computed using \cite{HT}. After a determinantal calculation, one finds that
$c_1(\nu_*\P^{\vee})=3g c_g/4$\   and  $\theta\cdot [W^2_d(C)]=c_g(g-d+5)/4$.
\end{proof}

\begin{proposition}\label{dinm7n} The class of the closure of $\cD_1$ in $\mm_{7, n}$ is given by the formula
$$\dd_1\equiv 44\lambda+6(\psi_1+\psi_2)-6\delta_{\mathrm{irr}}-28\delta_{0: 12}-6\sum_{j=3}^n (\delta_{0:1j}+\delta_{0:2j})-\cdots\in \mathrm{Pic}(\mm_{7, n}).$$
\end{proposition}
\begin{proof} We denote by $\phi_{12}:\mm_{7, n}\rightarrow \mm_{7, 2}$ the morphism retaining the first two marked points. Then $\dd_1=\phi_{12}^*(\overline{\mathfrak{Node}}_7)$, and the conclusion follows from  Proposition \ref{divm72} using the pull-back formulas for generators of $\mbox{Pic}(\mm_{7, 2})$, see e.g. \cite{Log} Theorem 2.3.
\end{proof}

\begin{theorem}
$\mm_{7, n}$ is uniruled for $n\leq 13$.
\end{theorem}
\begin{proof}
We start by constructing a covering curve for $\dd_1$. Choose general points $p_1, \ldots, p_8$, $x_3, \ldots, x_{12}\in \PP^2$, and a general line $l\subset \PP^2$. Then consider the pencil of plane septics of genus $7$ passing through $x_3, \ldots, x_{12}$ and having nodes at $p_1, \ldots, p_8$. Blowing-up the nodes
as well as the base points of the pencil, we obtain a fibration $f:S\rightarrow \PP^1$, where $S:=\mathrm{Bl}_{25}(\PP^2)$. We observe that $f$ has sections $\{E_{x_i}\}_{i=3, \ldots, 12}$, given by the respective effective divisors, a $2$-section given by $E_{p_1}$ and a $7$-section induced by the proper transform of $l$. We make base changes of order $2$ and $7$ respectively, to arrive at the $1$-cycle
$\Gamma_1:=\bigl\{[C_t, \ x_1(t), \ldots, x_{13}(t)]:t\in \PP^1\bigr\}\subset \mm_{7, 13}$,
where $x_1(t)$ and $x_2(t)$ map to the fixed node $p_1\in \PP^2$, whereas the image of $x_{13}(t)$ lies on the line $l$. Then one computes that:
$$
\Gamma_1 \cdot \lambda= 14\cdot g=98, \ \Gamma_1 \cdot \psi_1=\Gamma_1 \cdot \psi_2=35, \ \Gamma_1 \cdot \psi_{3}=\cdots =\Gamma_1 \cdot \psi_{{12}}=14,\ \Gamma_1\cdot \psi_{13}=24.
$$
Furthermore, $\Gamma_1 \cdot \delta_{0: 12}=14,\  \Gamma_1\cdot \delta_{0:\mathrm{irr}}=14\cdot 52=728$, and finally $\Gamma_1\cdot \delta_{j: T}=0$ for all pairs $(j, T)\neq \bigl(0, \{1, 2\}\bigr)$. Clearly  $\Gamma_1$ is a covering curve for $\dd_1$.
\vskip 3pt

Next, we construct a covering curve for $\dd_2$ and use that if $[C]\in \cM_{7, 4}^1$ and $A\in W^1_4(C)$ is the corresponding pencil, then there exists a point $p\in C$ such that $A\otimes \OO_C(3p)\in W^2_7(C)$. To reverse this construction, one fixes general points $p, \{p_i\}_{i=1}^5$, $\{x_j\}_{j=1}^{13} \in \PP^2$ and considers the pencil of genus $7$ septics with a $3$-fold point at $p$, nodes at
$p_1, \ldots, p_5$ and passing through $x_1, \ldots, x_{13}$. This induces a covering curve $\Gamma_2\subset \dd_2$ whose numerical invariants are as follows:
$$\Gamma_2\cdot \lambda=7, \ \Gamma_2\cdot \delta_{\mathrm{irr}}=53, \ \Gamma_2\cdot \psi_i=1 \mbox{ for } i=1, \ldots, 13, \ \ \Gamma_2\cdot \delta_{j: T}=0\ \mbox{ for all }\  (j, T).$$
One computes that  $\Gamma_1\cdot \dd_1=-28, \ \Gamma_2\cdot \dd_2=-14, \Gamma_1\cdot \dd_2=14, \ \Gamma_2\cdot \dd_1=28$,  as well as
$\Gamma_1\cdot K_{\mm_{7, 13}}=24,  \ \Gamma_2\cdot K_{\mm_{7, 13}}=-28$. The assumptions of Proposition \ref{uniruled} are thus fulfilled.
\end{proof}

\end{document}